\documentclass[oneside, 11pt, reqno]{amsart}
\usepackage[left=1in, right=1in, top=1in, bottom=1in]{geometry}

\usepackage[usenames,dvipsnames]{xcolor}
\usepackage{amsmath}

\usepackage{graphicx,enumitem}
\usepackage{mathtools}
\usepackage{extarrows} 
\usepackage{verbatim}

\usepackage{amssymb} 
\usepackage{soul}
\usepackage{cancel}
\allowdisplaybreaks	

\usepackage[margin=0.7cm]{caption}

\usepackage{multirow}

\newcommand{\Zitkovic}[1]{{\v Z}itkovi\'c}
\newcommand{\Sirbu}[1]{S\^\i rbu}

\numberwithin{equation}{section}
\theoremstyle{plain}                
\newtheorem{theorem}{Theorem}[section]
\newtheorem{lemma}[theorem]{Lemma}
\newtheorem{proposition}[theorem]{Proposition}
\newtheorem{corollary}[theorem]{Corollary}

\theoremstyle{definition}           

\theoremstyle{remark}
\newtheorem{remark}[theorem]{Remark}

\newcommand{\proref}[1]{Proposition~\ref{#1}}

\newcommand{\lemref}[1]{Lemma~\ref{#1}}

\DeclareMathOperator{\pre}{pre}
\DeclareMathOperator{\post}{post}
\DeclareMathOperator{\Mer}{Mer}
\DeclareMathOperator{\textnonit}{text}
\usepackage{graphicx} 

\usepackage[toc,title,page]{appendix}

\usepackage{fancyhdr}	
\begin{document}

\title{Impact Of Income And Leisure On Optimal Portfolio, Consumption, Retirement Decisions Under Exponential Utility}

\author{Tae Ung Gang$^{\ast}$}
\thanks{$^{\ast}$Stochastic Analysis and Application Research Center, Korea Advanced Institute of Science and Technology (email: gangtaeung@kaist.ac.kr).}
\author{Yong Hyun Shin$^{\ast\ast}$}
\thanks{$^{\ast\ast}$Department of Mathematics \& Research Institute of Natural Sciences, Sookmyung Women's University (email: yhshin@sookmyung.ac.kr).}

\maketitle

\begin{abstract}
We study an optimal control problem encompassing investment, consumption, and retirement decisions under exponential (CARA-type) utility. The financial market comprises a bond with constant drift and a stock following geometric Brownian motion. The agent receives continuous income, consumes over time, and has the option to retire irreversibly, gaining increased leisure post-retirement compared to pre-retirement. The objective is to maximize the expected exponential utility of weighted consumption and leisure over an infinite horizon. Using a martingale approach and dual value function, we derive implicit solutions for the optimal portfolio, consumption, and retirement time. The analysis highlights key contributions: first, the equivalent condition for no retirement is characterized by a specific income threshold; second, the influence of income and leisure levels on optimal portfolio, consumption, and retirement decisions is thoroughly examined. These results provide valuable insights into the interplay between financial and lifestyle choices in retirement planning.
\end{abstract}




\section{Introduction} \label{sec:introduction}

After the pioneering work of Merton \cite{M1, M2}, continuous-time optimal consumption and investment problems have become one of the central areas of mathematical finance, extensively studied under various economic conditions. One of the most significant themes in this field is the voluntary retirement problem, coupled with optimal portfolio selection (Choi and Shim \cite{CS}; Farhi and Panageas \cite{FP}; Choi et al. \cite{CSS}; Dybvig and Liu \cite{DL} etc.). These studies typically address the problem as an irreversible decision, formulated mathematically as a free boundary problem. Notably, most of the existing research has focused on power-type utility functions, including the Cobb-Douglas utility.

In this paper, we aim to advance the field by analyzing the optimization problem using a combination of exponential-type (CARA) utility and Cobb-Douglas utility. To the best of our knowledge, no prior literature has examined voluntary retirement decisions involving CARA utility, making this study a novel contribution to the area.

When addressing continuous-time optimal consumption and investment problems, researchers commonly rely on two mathematical approaches. The first approach involves deriving the Hamilton-Jacobi-Bellman (HJB) equation and obtaining the solution, known as the dynamic programming method (Karatzas et al. \cite{KLSS}). The second approach employs the Lagrangian method to obtain a dual solution, referred to as the dual/martingale method (Karatzas et al. \cite{KLS} and Cox and Huang \cite{CH}). In this study, we adopt the dual approach, integrating it with the optimal stopping problem framework (Karatzas and Wang \cite{KW}). This combination allows us to effectively model and analyze the decision-making process.

Our benchmark model builds upon the foundational work of Farhi and Panageas \cite{FP}. In their study, the authors considered an optimization problem involving power-type Cobb-Douglas utility, where the optimal leisure was represented by different constants before and after retirement. Inspired by their approach, we extend the analysis to include exponential-type Cobb-Douglas utility, where optimal leisure is characterized by two constants, $\underline{L} > 0$ before retirement and $\overline{L} > 0$ after retirement, with $\underline{L} < \overline{L}$. This setup enables us to examine the optimal solutions, particularly focusing on how labor income $Y > 0$ influences the voluntary retirement wealth threshold $\bar{x}$.

This paper makes two main contributions. First, we derive an equivalent condition under which individuals never retire from labor. This condition is expressed as an inequality, requiring that income $Y$ should exceed a specific threshold determined by the consumption weight parameter $\alpha$ and the difference in leisure, $\overline{L} - \underline{L}$. This result provides a clear economic interpretation and offers insights into how income and leisure preferences influence retirement decisions.

Second, we analytically demonstrate how optimal consumption, portfolio choices, and retirement timing are influenced by labor income $Y$, pre-retirement leisure $\underline{L}$, and post-retirement leisure $\overline{L}$. Specifically, we show that the retirement wealth threshold $\bar{x}$ increases with respect to $Y$ and $\underline{L}$ but decreases with respect to $\overline{L}$. Furthermore, we find that post-retirement optimal consumption and portfolio choices are independent of $(Y, \underline{L}, \overline{L})$, and pre-retirement optimal consumption and portfolio choices show opposite tendency of monotonicity in $\underline{L}, \overline{L}$. To be more specific, pre-retirement optimal consumption is non-decreasing in $\underline{L}$ and non-increasing in $\overline{L}$ whereas pre-retirement optimal portfolio is non-increasing in $\underline{L}$ and non-decreasing in $\overline{L}$. These findings deepen our understanding of the interplay between labor income, leisure preferences, and financial decisions in the context of voluntary retirement.

The rest of this paper is organized as follows: We introduce our model in Section \ref{sec:model}. Section \ref{sec:optimization} provides an optimization problem. We analyze the value function for post-retirement in Section \ref{sec:post_value} and then return to the optimization problem in Section \ref{sec:original}. In Section \ref{sec:solution}, we derive a solution to free boundary value problem. We recall a benchmark case (Merton problem) in Section \ref{sec:merton}, and provide verification and summary in Section \ref{sec:Verification}. In Section \ref{sec:properties}, we analyze some properties of the optimal plan and reach a conclusion in Section \ref{sec:conclusions}.


\section{Model} \label{sec:model}

Under the filtered probability space $(\Omega, \mathcal{F}, (\mathcal{F}_{t})_{t \geq 0}, \mathbb{P})$, which satisfies the usual conditions, we define a standard Brownian motion \((B_{t})_{t \geq 0}\). We consider a financial market comprising one bond and one stock, whose prices evolve according to the following ordinary differential equation (ODE) and stochastic differential equation (SDE), respectively:  
\begin{align*}
  d S_{t}^{(0)} = r\,S_{t}^{(0)} d t, \qquad d S_{t} = S_{t} (\mu d t + \sigma d B_{t}),
\end{align*}
where $r,~\mu$, and $\sigma$ are constants. Here, $r>0$ represents the risk-free rate, $\mu>r$ is the stock's drift rate, and $\sigma>0$ denotes the stock's volatility.

An agent in this market chooses a consumption rate $(c_{t})_{t \geq 0}$, which is non-negative and progressively measurable with respect to the filtration $(\mathcal{F}_{t})_{t \geq 0}$ and allocates $(\pi_{t})_{t \geq 0}$ to the stock at time $t$. The agent also earns a constant labor income $Y > 0$ per unit of time until retirement. The retirement decision is modeled as a one-time irreversible choice, represented by the stopping time $\tau$, which is $\mathcal{F}_{t}$-measurable. Once retired, the agent cannot return to work.

Beyond financial considerations, the agent derives utility from both consumption and leisure, and $(l_{t})_{t \geq 0}$ denotes a leisure rate, which differs before and after retirement. Specifically,  
\begin{align*}
  l_{t} = \begin{cases}
          \underline{L} \quad \text{ if } t < \tau, \\
          \overline{L} \quad \text{ if } t \geq \tau,
          \end{cases}
\end{align*}
where $0 < \underline{L} < \overline{L}$. This reflects the assumption that the agent experiences greater enjoyment from leisure after retirement.

The agent's consumption-leisure-portfolio plan is denoted by $(c, l, \pi)$. Let $X_{t} = X_{t}^{(c, l, \pi, x)}$ denote the agent's total wealth at time $t$ under the plan $(c, l, \pi)$, with initial wealth $X_{0} = x$. The wealth dynamics of the agent are given by the SDE:  
\begin{align} \label{X_dynamics}
  d X_{t} = \left\{ (\mu - r) \pi_{t} + r X_{t} - c_{t} + Y \cdot {\mathbf{1}}_{ \{ t \in [0, \tau) \}} \right\} d t + \sigma \pi_{t} d B_{t},
\end{align}
where the indicator function ${\mathbf{1}}_{\{t \in [0, \tau)\}}$ ensures that labor income $Y$ ceases upon retirement.

A consumption-leisure-portfolio plan $(c, l, \pi)$ is admissible if it satisfies the following conditions:  
\begin{equation} \label{admissible_condition}
  \begin{split}
  & \int_{0}^{t} c_{s} d s + \int_{0}^{t} \pi_{s}^{2} d s < \infty \quad \forall t \geq 0, \\
  & X_{t} \begin{cases}
          > - \tfrac{Y}{r} \quad \forall t \in [0, \tau), \\
          \geq 0 \quad \forall t \geq \tau.
          \end{cases}
  \end{split}
\end{equation}

The first condition ensures that consumption and the portfolio's squared allocation are integrable over any time horizon, guaranteeing finite wealth dynamics. The second condition imposes a borrowing constraint before retirement, restricting the agent's wealth to be greater than $-Y/r$, and requires non-negative wealth after retirement. These constraints reflect the agent's inability to sustain excessive borrowing or negative wealth in retirement.


\section{Optimization problem} \label{sec:optimization}

The agent seeks to determine an admissible consumption-leisure-portfolio plan \((c, l, \pi)\) that maximizes the following value function:  
\begin{align} \label{Original_problem}
  V(x) = \sup_{(c, \pi, \tau)} \mathbb{E}^{x} \left[ \int_{0}^{\infty} e^{- \beta s} U(c_{s}, l_{s}) d s \right],
\end{align}
where $\beta > 0$ is the subjective discount rate, $\mathbb{E}^{x}[\cdot]$ represents the conditional expectation generated by the initial condition $X_{0} = x$, and the utility function is defined as  
$$
U(c, l) := - \tfrac{e^{- \gamma^{*} \left\{\alpha c + (1 - \alpha) l\right\}}}{\alpha \gamma^{*}} = - \tfrac{e^{- \gamma c}}{\gamma} \cdot A(l).
$$
Here $\alpha \in (0, 1)$ is the weight parameter for consumption, $\gamma^{*} > 0$ is the agent's risk aversion, $\gamma := \alpha \gamma^{*}$, and $A(l) := e^{- \gamma^{*} (1 - \alpha) l}$. The objective reflects the agent's trade-off between consumption and leisure over time, discounted at the rate $\beta$.


\section{Value function for post-retirement} \label{sec:post_value}

For the post-retirement case, the wealth dynamics in \eqref{X_dynamics} reduce to  
\begin{align}
  d X_{t} & = \left\{ (\mu - r) \pi_{t} + r X_{t} - c_{t} \right\} d t + \sigma \pi_{t} d B_{t}, \label{X_dynamics_post}
\end{align}
where the condition $X_{t} \geq 0$ for all $t \geq 0$ should be satisfied. Let $\theta := (\mu - r)/\sigma > 0$ denote the Sharpe ratio, and define $H_{t} := e^{- \left( r + \frac{1}{2}\theta^{2}\right) t - \theta B_{t}} > 0$. Applying It\^{o}'s lemma to $H_{t} X_{t}$, we obtain  
\begin{align}
  \int_{0}^{t} H_{s} c_{s} d s + H_{t} X_{t} = x + \int_{0}^{t} H_{s} (\sigma \pi_{s} - \theta X_{s}) d B_{s}. \label{ito_for_HX_post}
\end{align}
The right-hand side of \eqref{ito_for_HX_post} is a bounded (from below) local martingale because the left-hand side is non-negative. By Fatou's lemma, the left-hand side in \eqref{ito_for_HX_post} is a supermartingale. Thus,  
\begin{align*}
  \mathbb{E}^{x} \left[ \int_{0}^{t} H_{s} c_{s} d s + H_{t} X_{t} \right] \leq x.
\end{align*}
If the condition $\lim_{t \to \infty} \mathbb{E} \left[ H_{t} X_{t} \right] = 0$ holds (to be verified later), then the monotone convergence theorem implies  
\begin{align}
  \mathbb{E}^{x} \left[ \int_{0}^{\infty} H_{s} c_{s} d s \right] \leq x. \label{Budget_constraint_post}
\end{align}

Let us define the post-retirement value function as  
\begin{align} \label{Post_value_define}
  V_{\post}(x) := \sup_{(c, \pi)} \mathbb{E}^{x} \left[ \int_{0}^{\infty} e^{- \beta s} \left( - \tfrac{e^{- \gamma c_{s}}}{\gamma} \cdot A(\overline{L}) \right) d s \right].
\end{align}

To solve \eqref{Post_value_define}, we apply the Lagrange multiplier method. For a fixed $y > 0$, the following inequality holds:  
\begin{equation}
  \begin{split}
  V_{\post}(x) - y x \leq \mathbb{E}^{y} \left[ \int_{0}^{\infty} e^{- \beta s} \overline{U} \left( y e^{\beta s} H_{s} \right) d s \right], \label{Lagrange}
  \end{split}
\end{equation}
where $\mathbb{E}^{y}[\cdot]$ denotes the conditional expectation given $y_{0} = y$. The inequality is derived using the budget constraint in \eqref{Budget_constraint_post}. The dual utility function \(\overline{U}(y)\) is defined as  
$$
\overline{U}(y) := \sup_{c \geq 0} \left( - \tfrac{e^{- \gamma c}}{\gamma} \cdot A(\overline{L}) - c y \right) = \tfrac{- y \ln \left( \frac{A(\overline{L})}{y} \right) - y}{\gamma} \cdot \mathbf{1}_{ \{ y \leq A(\overline{L}) \} } - \tfrac{A(\overline{L})}{\gamma} \cdot \mathbf{1}_{ \{ y > A(\overline{L}) \} }
$$
with the maximizer $c = \tfrac{1}{\gamma} \left( \ln \left( \frac{A(\overline{L})}{y} \right) \right)^{+}$, where the notation $(x)^{+} := \max(x, 0)$ ensures non-negativity. This result exploits the strict concavity of the map $c \mapsto - \tfrac{e^{- \gamma c}}{\gamma} \cdot A(\overline{L}) - c y$. Also we see that $\overline{U}$ is decreasing and convex.

Let us define the process  
\begin{align}
y_{t} := \lambda e^{\beta t} H_{t}, \quad \text{which evolves according to} \quad d y_{t} = y_{t} \big[ (\beta - r) dt - \theta d B_{t} \big]. \label{y_t}
\end{align} 
The dual function for the post-retirement case is given by  
\begin{align}
  \tilde{V}_{\post}(y) := \mathbb{E}^{y} \left[ \int_{0}^{\infty} e^{- \beta s} \overline{U}(y_{s}) d s \right] < 0. \label{Dual_function_post}
\end{align}
By the Feynman-Kac formula, the function \(\tilde{V}_{\post}(y)\) satisfies the ODE:  
\begin{align}
  \dfrac{1}{2}\theta^{2} y^{2}\, \tilde{V}_{\post}''(y) + (\beta - r) y\, \tilde{V}_{\post}'(y) - \beta\, \tilde{V}_{\post}(y) + \overline{U}(y) = 0 \label{ODE_post}
\end{align}
provided the conditions  
$$
\lim_{t \to \infty}  \mathbb{E} \big[e^{- \beta t}\tilde{V}_{\post}(y_{t})\big] = 0 \quad \text{and} \quad y_{t} \tilde{V}_{\post}'(y_{t}) \text{ is square integrable}
$$  
are satisfied. To derive this, we apply It\^{o}'s lemma to \(e^{-\beta t} \tilde{V}_{\post}(y_{t})\):  
\begin{align*}
  & d \left( e^{- \beta t} \tilde{V}_{\post}(y_{t}) \right) = e^{- \beta t} \left( \dfrac{1}{2}\theta^{2} y_{t}^{2} \tilde{V}_{\post}''(y_{t}) + (\beta - r) y_{t} \tilde{V}_{\post}'(y_{t}) - \beta \tilde{V}_{\post}(y_{t}) \right) d t - e^{- \beta t} \theta y_{t} \tilde{V}_{\post}'(y_{t}) d B_{t} \\
  & \qquad \qquad \qquad \quad \;\; = - e^{- \beta t} \overline{U}(y_{t}) d t - e^{- \beta t} \theta y_{t} \tilde{V}_{\post}'(y_{t}) d B_{t} \\
  & \Rightarrow \tilde{V}_{\post}(y) = \mathbb{E}^{y} \left[ \int_{0}^{t} e^{- \beta s} \overline{U}(y_{s}) d s \right] + e^{- \beta t} \mathbb{E}^{y} [\tilde{V}_{\post}(y_{t})] + \theta \mathbb{E}^{y} \left[ \int_{0}^{t} e^{- \beta s} y_{s} \tilde{V}'(y_{s}) d B_{s} \right] \xrightarrow[]{t \rightarrow \infty} \mathbb{E}^{y} \left[ \int_{0}^{\infty} e^{- \beta s} \overline{U}(y_{s}) d s \right].
\end{align*}

The quadratic equation  $f(m) := \theta^{2} m^{2} + (2 (\beta - r) - \theta^{2}) m - 2 \beta = 0$ has two real roots $m_{-}$ and $m_{+}$. Since $m \mapsto f(m)$ is strictly convex, $f(0) = - 2 \beta < 0$, and $f(1) = - 2 r < 0$, we get $m_{-} < 0 < 1 < m_{+}$. Furthermore, we can rewrite $\beta$ and $r$ using these roots:
\begin{align} \label{Relation_beta_r}
  \begin{cases}
  \beta = - \tfrac{\theta^{2}}{2} \cdot m_{+} m_{-}, \\
  r = - \tfrac{\theta^{2}}{2} \cdot (m_{+} - 1) (m_{-} - 1).
  \end{cases}
\end{align}
Now we solve the ODE \eqref{ODE_post} separately for the two regions $y \leq A(\overline{L})$ and $y > A(\overline{L})$ and use the above relation, we derive $C^{1}((0, \infty))$ solution 
\begin{equation}  \label{Post_value}
  \begin{split}
  \tilde{V}_{\post}(y) & = \begin{cases}
   \frac{m_{-} - 1}{\gamma r m_{+} (m_{+} - 1) (m_{+} - m_{-}) A(\overline{L})^{m_{+} - 1}} y^{m_{+}} + \left( \frac{\beta - r + \frac{\theta^{2}}{2}}{\gamma r^{2}} - \frac{\ln (A(\overline{L})) + 1}{\gamma r} \right) y + \frac{y \ln (y)}{\gamma r}, & \text{if } y \leq A(\overline{L}), \\
   \frac{m_{+} - 1}{\gamma r m_{-} (m_{-} - 1) (m_{+} - m_{-}) A(\overline{L})^{m_{-} - 1}} y^{m_{-}} - \frac{A(\overline{L})}{\gamma \beta}, & \text{if } y > A(\overline{L}),
   \end{cases} \\
  \tilde{V}_{\post}'(y) & = \begin{cases}
   \frac{m_{-} - 1}{\gamma r (m_{+} - 1) (m_{+} - m_{-})} \left( \tfrac{A(\overline{L})}{y} \right)^{1 - m_{+}} + \frac{\beta - r + \frac{\theta^{2}}{2}}{\gamma r^{2}} - \frac{\ln \left( \frac{A(\overline{L})}{y} \right)}{\gamma r}, & \text{if } y \leq A(\overline{L}), \\
   \frac{m_{+} - 1}{\gamma r (m_{-} - 1) (m_{+} - m_{-})} \left( \tfrac{A(\overline{L})}{y} \right)^{1 - m_{-}}, & \text{if } y > A(\overline{L}),
   \end{cases} \\
  y \tilde{V}_{\post}''(y) & = \begin{cases}
   \frac{m_{-} - 1}{\gamma r (m_{+} - m_{-})} \left( \tfrac{A(\overline{L})}{y} \right)^{1 - m_{+}} + \frac{1}{\gamma r}, & \text{if } y \leq A(\overline{L}), \\
   \frac{m_{+} - 1}{\gamma r (m_{+} - m_{-})} \left( \tfrac{A(\overline{L})}{y} \right)^{1 - m_{-}}, & \text{if } y > A(\overline{L}).
   \end{cases}
   \end{split}
\end{equation}
Let us now examine some key properties of $\tilde{V}_{\post}$ to support the verification argument. 
\begin{lemma} \label{Post_value_sign} {\ \\}
(i) $\tilde{V}_{\post} < 0$, $\tilde{V}_{\post}' < 0$, $\tilde{V}_{\post}'' > 0$. \\
(ii) $\tilde{V}_{\post} \in C^{3}((0, \infty))$. \\
(iii) $\displaystyle\sup_{y > 0} (|\tilde{V}_{\post}(y)| + |y \tilde{V}_{\post}'(y)| + |y^{2} \tilde{V}_{\post}''(y)|) < \infty$.
\end{lemma}
\begin{proof} {\ \\}
\noindent (i) Let us consider the function $h(y) := \frac{m_{-} - 1}{\gamma r (m_{+} - m_{-})} \left( \frac{A(\overline{L})}{y} \right)^{1 - m_{+}} + \frac{1}{\gamma r}$, which is strictly decreasing in $y$, with $h(0) = \tfrac{1}{\gamma r} > 0$ and $h(A(\overline{L})) = \tfrac{m_{+} - 1}{\gamma r(m_+-m_-)} > 0$. From this, we deduce that $\tilde{V}_{\post}''(y) > 0$ for all $y > 0$. Combining this with the fact that $\tilde{V}_{\post}'(A(\overline{L})) < 0$ and using \eqref{Post_value}, we conclude $\tilde{V}_{\post}'(y) < 0$ for all $y > 0$. Furthermore, since $\lim_{y \downarrow 0} \tilde{V}_{\post}(y) = 0$ and $\tilde{V}_{\post}'(y) < 0$ for all $y > 0$, it follows that $\tilde{V}_{\post}(y) < 0$ for all $y > 0$.

\noindent (ii) It is straightforward to verify that $\tilde{V}_{\post} \in C^{2}((0, \infty))$. Due to \eqref{ODE_post}, $\tilde{V}_{\post} \in C^{2}((0, \infty))$, and $\overline{U} \in C^{1}((0, \infty))$, we conclude that $\tilde{V}_{\post} \in C^{3}((0, \infty))$.

\noindent (iii) The finiteness of the supremum is a direct result of the boundedness of the terms.
\end{proof}
From the duality relation 
\begin{align*}
  V_{\post}(x) = \inf_{y > 0} \left(\tilde{V}_{\post}(y) + y x\right),
\end{align*}
we derive the first-order condition $X_{t} = - \tilde{V}_{\post}'(y_{t})$. The smoothness property \(\tilde{V}_{\post} \in C^{3}((0, \infty))\) allows us to apply It\^{o}'s lemma to obtain the dynamics of \(X_{t}\):
\begin{align*}
  d X_{t} = - \left( (\beta - r) y_{t} \tilde{V}_{\post}''(y_{t}) +\dfrac{1}{2}\theta^{2} y_{t}^{2} \tilde{V}_{\post}'''(y_{t}) \right) d t + \theta y_{t} \tilde{V}_{\post}''(y_{t}) d B_{t}.
\end{align*}
Comparing this with the dynamics of \(X_{t}\) in \eqref{X_dynamics_post}, we identify the portfolio as 
\begin{align*}
  \pi_{t} = \tfrac{\theta y_{t} \tilde{V}_{\post}''(y_{t})}{\sigma} > 0.
\end{align*}
\begin{lemma} \label{Post_vanish}
For the optimal total wealth $X_{t}^{\post} = - \tilde{V}_{\post}'(y_{t})$, we have\\
 $\lim_{t \rightarrow \infty} \mathbb{E}^{x} \left[ H_{t} X_{t}^{\post} \right] = 0$ and $\lim_{t \rightarrow \infty} \mathbb{E}^{x} [ e^{- \beta t} \tilde{V}_{\post}(y_{t}) ] = 0$.
\end{lemma}
\begin{proof}
Let $y = ( - \tilde{V}_{\post}' )^{- 1} (x)$, where $(\cdot)^{-1}$ denotes the inverse function, which exists and is well-defined by \lemref{Post_value_sign}. Using this, we have
\begin{align*}
  \mathbb{E}^{x} \left[ H_{t} X_{t}^{\post} \right] & = \mathbb{E}^{x} \left[ - H_{t} \tilde{V}_{\post}'(y_{t}) \right] \leq e^{- \beta t} y^{- 1} \sup_{z > 0} |z \tilde{V}_{\post}'(z)| \xrightarrow[]{t \to \infty} 0, \\
  \mathbb{E}^{x} [ e^{- \beta t} \tilde{V}_{\post}(y_{t}) ] & \leq e^{- \beta t} \sup_{z > 0} |\tilde{V}_{\post}(z)| \xrightarrow[]{t \to \infty} 0. 
\end{align*}
Both limits rely on the finiteness of $|\tilde{V}_{\post}(z)|$, $|z \tilde{V}_{\post}'(z)|$, and properties established in \lemref{Post_value_sign}.
\end{proof}

To summarize the post-retirement case, we present the following theorem:
\begin{theorem}
In the post-retirement case, the value function $V_{\post}$, its dual function $\tilde{V}_{\post}$, the total wealth $X^{\post}$, the optimal consumption rate $c^{\post}$, and the optimal portfolio $\pi^{\post}$ are given by 
\begin{equation} \label{post-retirement}
  \begin{split}
  & \tilde{V}_{\post}(y) = - \dfrac{1}{\gamma} \; \mathbb{E}^{y} \left[ \int_{0}^{\infty} e^{- \beta s} \left\{ y_{s} \left( \ln \left( \tfrac{A(\overline{L})}{y_{s}} \right) + 1 \right) \mathbf{1}_{ \{ y_{s} \leq A(\overline{L}) \} } + A(\overline{L}) \cdot \mathbf{1}_{ \{ y_{s} > A(\overline{L}) \} } \right\} d s \right], \\
  & V_{\post}(x) = \inf_{y > 0} \left(\tilde{V}_{\post}(y) + y x\right), \\
  & ( X_{t}^{\post}, \pi_{t}^{\post}, c_{t}^{\post} ) = \left( - \tilde{V}_{\post}'(y_{t}), \tfrac{\theta y_{t} \tilde{V}_{\post}''(y_{t})}{\sigma}, \tfrac{1}{\gamma} \left( \ln \left( \tfrac{A(\overline{L})}{y_{t}} \right) \right)^{+} \right),
  \end{split}
\end{equation}
where $y_{t} := y_{0} e^{\left( \beta - r - \frac{\theta^{2}}{2} \right) t - \theta B_{t}}$.
\end{theorem}


\section{Return to the original problem}  \label{sec:original}

Let $t \geq 0$, an admissible consumption rate $(c_{t})_{t \geq 0}$, and any $\mathcal{F}_{t}$-stopping time $\tau$ be given. Jeanblanc et al. \cite{Jeanblanc} states that  
\begin{align*}
  \mathbb{E}^{x} \left[ e^{- \beta \tau} V_{\post}(X_{\tau}) \right] = \mathbb{E}^{x} \left[ \int_{\tau}^{\infty} e^{- \beta s} U(c_{s}, \overline{L}) d s \right].
\end{align*}
This equivalence allows us to reformulate \eqref{Original_problem} as:  
\begin{align}
  V(x) & = \sup_{(c, \pi, \tau)} \mathbb{E}^{x} \left[ \int_{0}^{\tau} e^{- \beta s} \left( - \dfrac{e^{- \gamma c_{s}}}{\gamma} \cdot A(\underline{L}) \right) d s + e^{- \beta \tau} V_{\post}(X_{\tau}) \right]. \label{Original_problem_another}
\end{align}
Next, applying It\^{o}'s lemma to $H_{t} X_{t}$ (on $t < \tau$), we obtain: 
\begin{align}
  \int_{0}^{t} H_{s} c_{s} d s + H_{t} \left( X_{t} + \dfrac{Y}{r} \right) = x + \int_{0}^{t} H_{s} (\sigma \pi_{s} - \theta X_{s}) d B_{s} + Y \left( \dfrac{H_{t}}{r} + \int_{0}^{t} H_{s} d s \right). \label{ito_for_HX_pre}
\end{align}
The right-hand side of \eqref{ito_for_HX_pre} is a non-negative local martingale because \eqref{admissible_condition} holds and $\tfrac{H_{t}}{r} + \int_{0}^{t} H_{s} d s$ is a martingale (shown using the definition and the independent increments of Brownian motion). Thus, by Fatou's lemma, the left-hand side of \eqref{ito_for_HX_pre} is a supermartingale. Applying the optional sampling theorem, we derive: 
\begin{align}
  \mathbb{E}^{x} \left[ \int_{0}^{\tau} H_{s} c_{s} d s + H_{\tau} \left( X_{\tau} + \dfrac{Y}{r} \right) \right] \leq x + \dfrac{Y}{r} \label{Budget_constraint_pre}
\end{align}
(see Koo et al. \cite{KRS} for another approach to deriving budget constraints). Using a Lagrange multiplier approach, as in the post-retirement case, we combine \eqref{Original_problem_another} with \eqref{Budget_constraint_pre}. Together with \eqref{y_t} and the dual value for the post-retirement case, this leads to:  
\begin{align}
  &V(x) - y x \nonumber\\
  &\leq \sup_{\tau} \mathbb{E}^{y} \left[ \int_{0}^{\tau} e^{- \beta s} \left( - \dfrac{e^{- \gamma c_{s}}}{\gamma} \cdot A(\underline{L}) - c_{s} y e^{\beta s} H_{s} \right) d s + e^{- \beta \tau} \left( \tilde{V}_{\post}(y_{\tau}) - \dfrac{Y y_{\tau}}{r} \right) \right] + \dfrac{Y }{r}y \nonumber\\
  &=: \tilde{V}(y). \label{Tilde_V_upper}
\end{align}
Let us define the pre-retirement value function as  
\begin{align} \label{Pre_value_define}
  V_{\pre}(x) := \sup_{(c, \pi, \tau)} \mathbb{E}^{x} \left[ \int_{0}^{\tau} e^{- \beta s} \left( - \dfrac{e^{- \gamma c_{s}}}{\gamma} \cdot A(\underline{L}) \right) d s \right].
\end{align}
To solve \eqref{Pre_value_define}, we consider the dual function for pre-retirement 
\begin{align}
  \tilde{V}_{\pre}(y) := \sup_{\tau} \mathbb{E}^{y} \left[ \int_{0}^{\tau} e^{- \beta s} \underline{U}(y_{s}) d s \right] + \dfrac{Y}{r} y, \label{Dual_function_pre}
\end{align}
where $\underline{U}(y) := \sup_{c \geq 0} \left( - \tfrac{e^{- \gamma c}}{\gamma} \cdot A(\underline{L}) - c y \right) = \tfrac{- y \ln \left( \frac{A(\underline{L})}{y} \right) - y}{\gamma} \cdot \mathbf{1}_{ \{ y \leq A(\underline{L}) \} } - \tfrac{A(\underline{L})}{\gamma} \cdot \mathbf{1}_{ \{ y > A(\underline{L}) \} }$ with the maximizer $c = \tfrac{1}{\gamma} \left( \ln \left( \frac{A(\underline{L})}{y} \right) \right)^{+}$ since the map $c \mapsto - \tfrac{e^{- \gamma c}}{\gamma} \cdot A(\underline{L}) - c y$ is strictly concave. Also we see that $\underline{U}$ is decreasing and convex.
\begin{lemma} \label{Pre_value_sign}
Let $\tau = \inf \{ s > 0 : y_s \leq \overline{y} \}$ be a maximizer for \eqref{Dual_function_pre}. Then   \\
(i) $\tilde{V}_{\pre}(y) \leq \tfrac{Y}{r} y$, $\tilde{V}_{\pre}'(y) < \tfrac{Y}{r}$, $\tilde{V}_{\pre}''(y) \geq 0$ for all $y > 0$. Equalities become strict if $y > \overline{y}$. \\
(ii) $\tilde{V}_{\pre} \in C((0, \infty)) \cap C^{2}((\overline{y}, \infty))$. \\
(iii) $\displaystyle\sup_{y > 0} \left(|\tilde{V}_{\pre}(y)| + |y \tilde{V}_{\pre}'(y)| + |y^{2} \tilde{V}_{\pre}''(y)|\right) < \infty$.
\end{lemma}
\begin{proof} {\ \\}
\noindent (i) Let us define $\tilde{v}_{\pre}(y) := \tilde{V}_{\pre}(y) - \frac{Y}{r} y$. Then $\tilde{v}_{\pre}$ satisfies  
\begin{align} \label{ODE_pre}
    \begin{cases}
  \frac{\theta^{2} y^{2}}{2} \cdot \tilde{v}_{\pre}''(y) + (\beta - r) y \tilde{v}_{\pre}'(y) - \beta \tilde{v}_{\pre}(y) + \underline{U}(y) = 0, &\mbox{for~} y > \overline{y}, \\
  \tilde{v}_{\pre}(y) = 0, &\mbox{for~} 0 < y \leq \overline{y}.
    \end{cases}
\end{align}
Thus we have
\begin{equation}  \label{Pre_value}
  \begin{split}
  \tilde{V}_{\pre}(y) & = \frac{Y }{r}y + \begin{cases}
   C_{1} y^{m_{+}} + C_{2} y^{m_{-}} + \left( \frac{\beta - r + \frac{\theta^{2}}{2}}{\gamma r^{2}} - \frac{\ln (A(\underline{L})) + 1}{\gamma r} \right) y + \frac{y \ln (y)}{\gamma r}, &\ \text{if } \overline{y} < y \leq A(\underline{L}), \\
   C_{3} y^{m_{-}} - \frac{A(\underline{L})}{\gamma \beta}, & \text{if } y > A(\underline{L}),
   \end{cases} \\
  \tilde{V}_{\pre}'(y) & = \frac{Y}{r} + \begin{cases}
   C_{1} m_{+} y^{m_{+} - 1} + C_{2} m_{-} y^{m_{-} - 1} + \frac{\beta - r + \frac{\theta^{2}}{2}}{\gamma r^{2}} + \frac{\ln \left(\frac{y}{A(\underline{L})}\right)}{\gamma r}, & \text{if } \overline{y} < y \leq A(\underline{L}), \\
   C_{3} m_{-} y^{m_{-} - 1}, & \text{if } y > A(\underline{L}),
   \end{cases} \\
  y \tilde{V}_{\pre}''(y) & = \begin{cases}
   C_{1} m_{+} (m_{+} - 1) y^{m_{+} - 1} + C_{2} m_{-} (m_{-} - 1) y^{m_{-} - 1} + \frac{1}{\gamma r}, & \text{if } \overline{y} < y \leq A(\underline{L}), \\
   C_{3} m_{-} (m_{-} - 1) y^{m_{-} - 1}, & \text{if } y > A(\underline{L}).
   \end{cases}
   \end{split}
\end{equation}
The coefficients $C_1$, $C_2$, and $C_3$ are determined by ensuring $\tilde{V}_{\pre} \in C^1((\overline{y}, \infty))$ and $\tilde{V}_{\pre}$ is continuous at $y = \overline{y}$:  
\begin{equation}  \label{Pre_coefficient}
    \begin{split}
  C_{1} & = \frac{m_{-} - 1}{\gamma r m_{+} (m_{+} - 1) (m_{+} - m_{-}) A(\underline{L})^{m_{+} - 1}} < 0, \\
  C_{2} & = \frac{r \ln \left(\frac{A(\underline{L})}{\overline{y}}\right) - \left( \beta - 2 r + \frac{\theta^{2}}{2} \right)}{\gamma r^{2} \overline{y}^{m_{-} - 1}} - C_{1} \overline{y}^{m_{+} - m_{-}} > 0, \\
  C_{3} & = C_{2} + \frac{m_{+} - 1}{\gamma r m_{-} (m_{-} - 1) (m_{+} - m_{-}) A(\underline{L})^{m_{-} - 1}} > \frac{m_{+} - 1}{\gamma r m_{-} (m_{-} - 1) (m_{+} - m_{-}) A(\underline{L})^{m_{-} - 1}} > 0.
    \end{split}
\end{equation}
Here, the proof of all signs in \eqref{Pre_coefficient} is complete if $C_{2} > 0$ is shown, which can be easily derived from $\tilde{V}_{\Mer}(y) < \tilde{V}_{\pre}(y)$ for $y > \overline{y}$ by considering the stochastic representation (refer to Section \ref{sec:merton} for the form of $\tilde{V}_{\Mer}$).

The function $ k(y) := C_{1} m_{+} (m_{+} - 1) y^{m_{+} - 1} + C_{2} m_{-} (m_{-} - 1) y^{m_{-} - 1} + \frac{1}{\gamma r} $ satisfies $ k' < 0$ with $k(A(\underline{L}))=\tfrac{m_{+} - 1}{\gamma r(m_+-m_-)}+C_{2} m_{-} (m_{-} - 1) A(\underline{L})^{m_{-} - 1} > 0$ due to $C_2>0$. This result and $C_3>0$ imply $\tilde{V}_{\pre}''(y) > 0$ for $y>\overline{y}$. Also, combining the convexity with $C_{3} > 0$ and $\tilde{v}_{\pre}'(A(\underline{L}))=\tilde{V}_{\pre}'(A(\underline{L})) - \frac{Y}{r} = C_{3} m_{-} A(\underline{L})^{m_{-} - 1} < 0$ gives $\tilde{v}_{\pre}'(y)=\tilde{V}_{\pre}'(y) - \frac{Y}{r}<0$ for $y > \overline{y}$. Similarly, combining this result with $\tilde{v}_{\pre}(\overline{y})=\tilde{V}_{\pre}(\overline{y}) - \frac{Y}{r}\overline{y} = 0$ gives $\tilde{v}_{\pre}(y)=\tilde{V}_{\pre}(y) - \frac{Y}{r}y<0$ for $y > \overline{y}$, that is, $\tilde{V}_{\pre}(y) < \frac{Y}{r}y$ for $y > \overline{y}$.

\noindent (ii) From \eqref{ODE_pre}, the fact that $\tilde{V}_{\pre} \in C^{1}((\overline{y}, \infty))$, and $\underline{U} \in C^{1}((\overline{y}, \infty))$, we conclude that $\tilde{V}_{\pre} \in C^{2}((\overline{y}, \infty))$. $\tilde{V}_{\pre} \in C((0, \infty))$ is already shown by determining $C_{1}$ ,$C_{2}$, $C_{3}$.

\noindent (iii)  The finiteness of the supremum is a direct result of the boundedness of the terms.
\end{proof}

Similarly to the post-retirement case, we can establish the following results if the optimal retirement time $\tau^{*}$ is identified:
\begin{theorem}
Let $\tau^{*}$ be the optimal retirement time. In the pre-retirement case, the value function $V_{\pre}$, its dual function $\tilde{V}_{\pre}$, the total wealth $X^{\pre}$, the optimal consumption rate $c^{\pre}$, and the optimal portfolio $\pi^{\pre}$ are given by
\begin{equation} \label{pre-retirement}
  \begin{split}
  & \tilde{V}_{\pre}(y) = - \frac{1}{\gamma} \; \mathbb{E}^{y} \left[ \int_{0}^{\tau^{*}} e^{- \beta s} \left\{ y_{s} \left( \ln \left( \frac{A(\underline{L})}{y_{s}} \right) + 1 \right) \mathbf{1}_{ \{ y_{s} \leq A(\underline{L}) \} } + A(\underline{L}) \cdot \mathbf{1}_{ \{ y_{s} > A(\underline{L}) \} } \right\} d s \right] + \frac{Y y}{r}, \\
  & V_{\pre}(x) = \inf_{y > 0} \left(\tilde{V}_{\pre}(y) + y x\right), \\
  & ( X_{t}^{\pre}, \pi_{t}^{\pre}, c_{t}^{\pre} ) = \left( - \tilde{V}_{\pre}'(y_{t}), \frac{\theta y_{t} \tilde{V}_{\pre}''(y_{t})}{\sigma}, \frac{1}{\gamma} \left( \ln \left( \tfrac{A(\underline{L})}{y_{t}} \right) \right)^{+} \right).
  \end{split}
\end{equation}
\end{theorem}
This formulation mirrors the structure used in the post-retirement scenario, with adjustments to account for the pre-retirement context. The components are defined in terms of the dual function and the optimal stopping time, ensuring consistency with the economic framework.


\section{Solution to free boundary value problem} \label{sec:solution}

Under considering \eqref{Tilde_V_upper}, let us define 
\begin{align*}
  \tilde{v}(y) & := \sup_{\tau} \mathbb{E}^{y} \left[ \int_{0}^{\tau} e^{- \beta s} \underline{U}(y_{s}) d s + e^{- \beta \tau} \left( \tilde{V}_{\post}(y_{\tau}) - \tfrac{Y }{r}y_{\tau} \right) \right].
\end{align*}
From \eqref{Original_problem_another} and \eqref{pre-retirement}, it follows that $\tilde{v}(y)= \tilde{V}(y) - \tfrac{Y}{r}y$. Solving $\tilde{v}$ is equivalent to finding $\overline{y} \geq 0$ and a function $\tilde{w} : [0, \infty) \rightarrow \mathbb{R}$ satisfying the following  variational inequality (see p. 232 of Oksendal \cite{Oksendal1}):
\begin{align}
  \begin{cases}
  \frac{1 }{2}\theta^{2}y^{2} \tilde{w}''(y) + (\beta - r) y \tilde{w}'(y) - \beta \tilde{w}(y) + \underline{U}(y) - \overline{U}(y) + Y y = 0, & \mbox{for}~ y > \overline{y}, \\
  \frac{1 }{2}\theta^{2}y^{2} \tilde{w}''(y) + (\beta - r) y \tilde{w}'(y) - \beta \tilde{w}(y) + \underline{U}(y) - \overline{U}(y) + Y y \leq 0, & \mbox{for}~ 0 < y \leq \overline{y}, \\
  \tilde{w}(y) \geq 0, & \mbox{for}~ y > \overline{y}, \\
  \tilde{w}(y) = 0, & \mbox{for}~ 0 < y \leq \overline{y},
  \end{cases}  
   \label{ODE_equivalent}
\end{align}
where $\tilde{w}(y) = \tilde{v}(y) - \tilde{V}_{\post}(y) + \tfrac{Y}{r} y$, as derived from \eqref{ODE_post}. The stochastic representation of \eqref{ODE_equivalent} is given by
\begin{align}
  \tilde{w}(y) & = \sup_{\tau} \mathbb{E}^{y} \left[ \int_{0}^{\tau} e^{- \beta s} \left(\underline{U}(y_{s}) - \overline{U}(y_{s}) + Y y_{s}\right) d s \right] \label{ODE_equivalent_stochastic}
\end{align}
with $\tau = \inf \{ s \geq 0 : y_{s} \leq \overline{y} \}$  and $\overline{y} \geq 0$ is to be determined.

Note that $A(\overline{L}) < A(\underline{L})$ and 
\begin{align}
 0> \underline{U}(y) - \overline{U}(y) = \begin{cases}
                                       - \tfrac{1 - \alpha}{\alpha} \cdot (\overline{L} - \underline{L}) y, & \text{if } y \leq A(\overline{L}), \\
                                       \tfrac{y}{\gamma} \cdot \ln \left( \tfrac{y}{A(\underline{L})} \right) + \tfrac{A(\overline{L}) - y}{\gamma}, & \text{if } A(\overline{L}) < y \leq A(\underline{L}), \\
                                       \tfrac{A(\overline{L}) - A(\underline{L})}{\gamma}, &  \text{if } y > A(\underline{L}).
  \end{cases} \label{uU-oU}
\end{align}
By observing this function \eqref{uU-oU}, the map $y \mapsto \underline{U}(y) - \overline{U}(y)$ is continuously differentiable, and its derivative $y \mapsto \underline{U}'(y) - \overline{U}'(y)$ is bounded and non-decreasing. This implies that $y \mapsto \underline{U}(y) - \overline{U}(y)$ is $C^{1}$, Lipschitz continuous, and convex.

To solve \eqref{ODE_equivalent} for $y > \overline{y}$, we use the variation of parameters method, yielding  
\begin{align}
  \tilde{w}(y) = \begin{cases}
   C y^{m_{-}} + \tfrac{2}{\theta^{2} (m_{+} - m_{-})} \cdot \bigg[ y^{m_{+}} \int_{y}^{\infty} z^{- 1 - m_{+}} \left\{\underline{U}(z) - \overline{U}(z) + Y z\right\} d z &\\
   \qquad + y^{m_{-}} \int_{0}^{y} z^{- 1 - m_{-}} \left\{\underline{U}(z) - \overline{U}(z) + Y z\right\} d z \bigg], 
   & \text{if } y > \overline{y}, \\
   0, & \text{if } 0 < y \leq \overline{y},
   \end{cases} \label{Tilde_y}
\end{align}
where $C$ and $\overline{y}$ are to be determined. Assuming $\tilde{w} \in C^{1}((0, \infty))$ so that we apply the boundary conditions $\tilde{w}(\overline{y}) = \tilde{w}'(\overline{y}) = 0$, we derive the following system:  
\begin{equation*}
  \begin{split}
  0 & = C \overline{y}^{m_{-}} + \tfrac{2}{\theta^{2} (m_{+} - m_{-})}\bigg[ \overline{y}^{m_{+}} \int_{\overline{y}}^{\infty} z^{- 1 - m_{+}} \left\{\underline{U}(z) - \overline{U}(z) + Y z\right\} d z \\
  &\qquad+ \overline{y}^{m_{-}} \int_{0}^{\overline{y}} z^{- 1 - m_{-}} \left\{\underline{U}(z) - \overline{U}(z) + Y z\right\} d z \bigg], \\
  0 & = m_{-} C \overline{y}^{m_{-} - 1} + \tfrac{2}{\theta^{2} (m_{+} - m_{-})} \bigg[ m_{+} \overline{y}^{m_{+} - 1} \int_{\overline{y}}^{\infty} z^{- 1 - m_{+}} \left\{\underline{U}(z) - \overline{U}(z) + Y z\right\} d z \\
  & \qquad + m_{-} \overline{y}^{m_{-} - 1} \int_{0}^{\overline{y}} z^{- 1 - m_{-}} \left\{\underline{U}(z) - \overline{U}(z) + Y z\right\} d z \bigg].
  \end{split}
\end{equation*}
From these equations, we deduce the following:
\begin{equation} \label{Equation_overline_y}
  \begin{split}
  0 & = \int_{\overline{y}}^{\infty} z^{- 1 - m_{+}} \left\{\underline{U}(z) - \overline{U}(z) + Y z\right\} d z =: R(\overline{y}; Y), \\
  C & = - \frac{2}{\theta^{2} (m_{+} - m_{-})} \cdot \int_{0}^{\overline{y}} z^{- 1 - m_{-}} \left\{\underline{U}(z) - \overline{U}(z) + Y z\right\} d z.
  \end{split}
\end{equation}
The system implies that if income $Y$ is sufficiently large, the agent will not retire. The following lemma provides a threshold for $Y$ beyond which retirement is not optimal. 
\begin{proposition} \label{overline_y_zero_condition}
$\overline{y} < \infty$. In particular, $\overline{y} = 0$ if $Y \geq \tfrac{1 - \alpha}{\alpha} (\overline{L} - \underline{L})$.
\end{proposition}
\begin{proof}
If $\tau = 0$ (equivalently, $\overline{y} = \infty$), it is a candidate maximizer for \eqref{ODE_equivalent_stochastic}. To prove that $\tau = 0$ is not optimal, it suffices to demonstrate that the expectation in \eqref{ODE_equivalent_stochastic} is strictly positive for some stopping time $\tau > 0$.

If $Y \geq \tfrac{1 - \alpha}{\alpha}(\overline{L} - \underline{L})$, then $\underline{U}(y) - \overline{U}(y) + Y y \geq 0$ for all $y \leq A(\overline{L})$. Since
\begin{align*}
  \underline{U}'(y) - \overline{U}'(y) + Y = \begin{cases}
                               		Y - \tfrac{1 - \alpha}{\alpha}(\overline{L} - \underline{L}), & \text{if } y \leq A(\overline{L}),\\
                                       Y + \tfrac{1}{\gamma} \cdot \ln \left( \tfrac{y}{A(\underline{L})}\right), & \text{if } A(\overline{L}) < y \leq A(\underline{L}), \\
                                       Y, & \text{if } y > A(\underline{L})
  \end{cases}
\end{align*}
and $y \mapsto \tfrac{1}{\gamma} \cdot \ln \left( \tfrac{y}{A(\underline{L})} \right)$ is strictly increasing, we have $\underline{U}'(y) - \overline{U}'(y) + Y > 0$ for all $y > A(\overline{L})$. Therefore, $\underline{U}(y) - \overline{U}(y) + Y y > 0$ for all $y > A(\overline{L})$, and the integral \eqref{ODE_equivalent_stochastic} is positive. This leads to $\tau = \infty$, and to satisfy this condition, $\overline{y} = 0$.

If $Y < \tfrac{1 - \alpha}{\alpha} (\overline{L} - \underline{L})$, then  the convexity of $y \mapsto \underline{U}(y) - \overline{U}(y)$ ensures the existence of a unique $j \in (A(\overline{L}), \infty)$ such that 
\begin{align}
  \begin{cases}
    \underline{U}(y) - \overline{U}(y) + Y y < 0, & \text{if } y < j, \\
    \underline{U}(y) - \overline{U}(y) + Y y = 0, & \text{if } y = j, \\
    \underline{U}(y) - \overline{U}(y) + Y y > 0, & \text{if } y > j,
  \end{cases} \label{Sign_according_to_j}
\end{align}
(see Figure \ref{uU-oU_Yy} for an illustration of this behavior). Define the stopping time $\tau = \inf \{ s \geq 0 : y_{s} \leq j \}$. For this choice of $\tau$, $\tilde{w}(y) > 0$, which implies that $\overline{y} < \infty$.  
\end{proof}
The following proposition establishes the existence and uniqueness of $\overline{y}$.
\begin{proposition} \label{overline_y<j}
If $Y < \tfrac{1 - \alpha}{\alpha}(\overline{L} - \underline{L})$, the solution $\overline{y}$ of \eqref{Equation_overline_y} exists and is unique satisfying $\overline{y} \in (0, j)$ where $j$ is defined in \eqref{Sign_according_to_j}.
\end{proposition}
\begin{proof}
If $Y < \tfrac{1 - \alpha}{\alpha} \cdot (\overline{L} - \underline{L})$, we observe that $\lim_{y \downarrow 0} R(y; Y) < 0 < R(j; Y)$. By the intermediate value theorem, there exists $\overline{y} \in (0, j)$ such that $\overline{y}$ solves \eqref{Equation_overline_y}. The function $y \mapsto R(y; Y)$ is strictly increasing in $(0, j)$, ensuring that \(\overline{y}\) is unique. Additionally, for $y \geq j$, $R(y; Y) > 0$, which implies that $\overline{y} \in [j, \infty)$ is impossible.
\end{proof}
Using \proref{overline_y_zero_condition} and \proref{overline_y<j}, we deduce the following corollary:
\begin{corollary}
$\overline{y} = 0$ (i.e the agent never retires) if and only if $Y \geq \tfrac{1 - \alpha}{\alpha}(\overline{L} - \underline{L})$.
\end{corollary}

We observe that $y \mapsto R(y; Y)$ is continuous in $(0, \infty)$. Since $\underline{U}(y) - \overline{U}(y)$ is independent of $Y$ and $Y \mapsto R(y; Y)$ is continuous and strictly increasing with 
\begin{align*}
  \lim_{Y \downarrow 0} R(y; Y) < 0 < \lim_{Y \rightarrow \infty} R(y; Y)
\end{align*}
for fixed $y \in (0, \infty)$ (these limits are valid because the monotone convergence theorem allows us to exchange the limits with the integral). As a result, we conclude that  ``\textit{$Y \mapsto \overline{y}$ is a strictly decreasing function for $Y \in \left( 0, \tfrac{1 - \alpha}{\alpha}(\overline{L} - \underline{L}) \right)$}''. In particular, the intermediate value theorem guarantees the existence of unique values $0 < Y_{2} < Y_{1} < \infty$ such that 
\begin{align*}
  0 = R(A(\overline{L}); Y_{1}) = R(A(\underline{L}); Y_{2}).
\end{align*}

\begin{figure}[t]
		\begin{center}
			\includegraphics[width=0.45\textwidth]{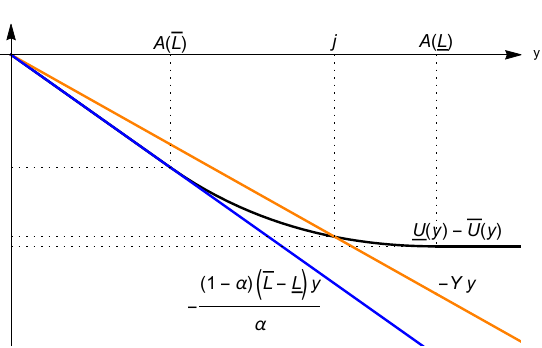}
	\end{center}
	 \caption{The graph shows $\underline{U}(y) - \overline{U}(y)$ (black), $-Y y$ (orange), $- \tfrac{1 - \alpha}{\alpha} \cdot (\overline{L} - \underline{L}) y$ (blue), and the intersection of black and orange (denoted by $j$) as functions of $y$. In this case, $Y < \tfrac{1 - \alpha}{\alpha} \cdot (\overline{L} - \underline{L})$.
		}
		\label{uU-oU_Yy}
\end{figure}

The location of $\overline{y}$ depends on the value of $Y$. Specifically,
\begin{align}
  \overline{y}= \; \begin{cases}
                0, & \text{if} \quad Y \geq \tfrac{1 - \alpha}{\alpha} (\overline{L} - \underline{L}), \\
                y_{1} \in (0, A(\overline{L})), & \text{if} \quad Y_{1} < Y < \tfrac{1 - \alpha}{\alpha}  (\overline{L} - \underline{L}), \\
                A(\overline{L}), & \text{if} \quad Y = Y_{1}, \\
                y_{2} \in (A(\overline{L}), A(\underline{L})), & \text{if} \quad Y_{2} < Y < Y_{1}, \\
                 A(\underline{L}), & \text{if} \quad Y = Y_{2}, \\
                 \tfrac{(m_{+} - 1) [A(\underline{L}) - A(\overline{L})]}{\gamma m_{+} Y}, & \text{if} \quad 0 < Y < Y_{2},
  \end{cases} \label{overline_y}
\end{align}
where $y_{1}$ and $y_{2}$ satisfy the equations
\begin{equation}  \label{oy_small_equation}
  \begin{split}
  y_1 & = \left[\dfrac{A(\underline{L})^{1-m_+}-A(\overline{L})^{1-m_+}}{\gamma m_{+} (m_{+} - 1)\left\{ Y - \frac{1 - \alpha}{\alpha} (\overline{L} - \underline{L}) \right\}}\right]^{\frac{1}{1-m_+}}\\
  0 & = \left[ \dfrac{Y}{m_{+} - 1} + \dfrac{2 - m_{+}}{\gamma (m_{+} - 1)^{2}}-\dfrac{ \ln \left( \frac{A(\underline{L})}{y_{2}} \right)}{\gamma (m_{+} - 1)} \right] \left( \frac{A(\underline{L})}{y_{2}} \right)^{m_{+} - 1} + \frac{A(\overline{L})}{\gamma m_{+} A(\underline{L})} \left( \dfrac{A(\underline{L})}{y_{2}} \right)^{m_{+}} - \dfrac{1}{\gamma m_{+} (m_{+} - 1)^{2}}.
  \end{split}
\end{equation}
Finally, it remains to verify that $\tilde{w}(y)$ satisfies \eqref{ODE_equivalent}. This step completes the analysis of the problem. 
\begin{proposition}
The function $\tilde{w}$, introduced in \eqref{Tilde_y}, belongs to $C^{1}((0, \infty)) \cap C^{2}((0, \infty) \setminus \{ \overline{y} \})$ and solves \eqref{ODE_equivalent}.
\end{proposition}
\begin{proof}
Since \eqref{Tilde_y} satisfies both equalities in \eqref{ODE_equivalent}, it suffices to verify the following conditions:
\begin{align}
  &\tfrac{1}{2}\theta^{2} y^{2} \tilde{w}''(y) + (\beta - r) y \tilde{w}'(y) - \beta \tilde{w}(y) + \underline{U}(y) - \overline{U}(y) + Y y \leq 0,  &~ \mbox{for~} 0 < y \leq \overline{y}, \label{Check_1} \\
  &\tilde{w}(y)\geq 0, &~  \mbox{for~} y > \overline{y}. \label{Check_2}
\end{align}
To verify \eqref{Check_1}, first observe that $\tilde{w}(y) = 0$ for $y \leq \overline{y}$. Substituting this into \eqref{Check_1}, the inequality reduces to $\underline{U}(y) - \overline{U}(y) + Y y \leq 0$ for $y \leq \overline{y}$. This holds true by combining \eqref{Sign_according_to_j} with \proref{overline_y<j}.

 Next, to verify \eqref{Check_2}, restrict $y > \overline{y}$ and rewrite \eqref{Tilde_y} using the value of $C$ in \eqref{Equation_overline_y} as follows:
\begin{align*}
   \tfrac{\theta^{2} (m_{+} - m_{-}) \tilde{w}(y)}{2 y^{m_{-}}} &= y^{m_{+} - m_{-}} \int_{y}^{\infty} z^{- 1 - m_{+}} \left\{\underline{U}(z) - \overline{U}(z) + Y z\right\} d z + \int_{\overline{y}}^{y} z^{- 1 - m_{-}} \left\{\underline{U}(z) - \overline{U}(z) + Y z\right\} d z \\
  & \Rightarrow \tfrac{d}{d y} \left( \tfrac{\theta^{2} \tilde{w}(y)}{2 y^{m_{-}}} \right) = y^{m_{+} - m_{-} - 1} \int_{y}^{\infty} z^{- 1 - m_{+}} \left\{\underline{U}(z) - \overline{U}(z) + Y z\right\} d z > 0,
\end{align*}
where the positivity follows from \eqref{Equation_overline_y} and \eqref{Sign_according_to_j}. Since $\tilde{w}(\overline{y}) = 0$, the above implies that $\tilde{w}(y) > 0 $ for $ y > \overline{y}$.
\end{proof}


\section{Benchmark case (Merton problem)} \label{sec:merton}

We consider the Merton problem to compare this with our model. In this setup, the agent has leisure $\underline{L}$ and income $Y$ throughout her lifetime (i.e., there is no retirement). The total wealth process $X_{t}$ and the value function $V_{\Mer}$ are defined as follows:  
\begin{align*}
  d X_{t} & = \left\{ (\mu - r) \pi_{t} + r X_{t} - c_{t} + Y \right\} d t + \sigma \pi_{t} d B_{t}, \\
  V_{\Mer}(x) & = \sup_{(c, \pi)} \mathbb{E}^{x} \left[ \int_{0}^{\infty} e^{- \beta s} \left( - \tfrac{e^{- \gamma c_{s}}}{\gamma} \cdot A(\underline{L}) \right) d s \right].
\end{align*}
Using a similar approach to the pre-retirement and post-retirement cases, we derive the following lemma:
\begin{lemma} \label{Mer_value_sign}
For the benchmark case, the value function $V_{\Mer}$, its dual $\tilde{V}_{\Mer}$, the total wealth $X_{t}^{\Mer}$, optimal consumption $c_{t}^{\Mer}$, and optimal portfolio $\pi_{t}^{\Mer}$ are given by
\begin{equation*}
  \begin{split}
  & \tilde{V}_{\Mer}(y) = - \tfrac{1}{\gamma} \; \mathbb{E}^{y} \left[ \int_{0}^{\infty} e^{- \beta s} \left\{ y_{s} \left( \ln \left( \tfrac{A(\underline{L})}{y_{s}} \right) + 1 \right) \mathbf{1}_{ \{ y_{s} \leq A(\underline{L}) \} } + A(\underline{L}) \cdot \mathbf{1}_{ \{ y_{s} > A(\underline{L}) \} } \right\} d s \right] + \tfrac{Y}{r} y, \\
  & V_{\Mer}(x) = \inf_{y > 0} (\tilde{V}_{\Mer}(y) + y x), \\
  & ( X_{t}^{\Mer}, \pi_{t}^{\Mer}, c_{t}^{\Mer} ) = \left( - \tilde{V}_{\Mer}'(y_{t}), \tfrac{\theta y_{t} \tilde{V}_{\Mer}''(y_{t})}{\sigma}, \tfrac{1}{\gamma} \left( \ln \left( \tfrac{A(\underline{L})}{y_{t}} \right) \right)^{+} \right).
  \end{split}
\end{equation*}
The closed-form expressions for $\tilde{V}_{\Mer}$, $\tilde{V}_{\Mer}'$, $\tilde{V}_{\Mer}''$ are 
\begin{equation}  \label{Mer_value}
  \begin{split}
  \tilde{V}_{\Mer}(y) & = \tfrac{Y}{r} y + \begin{cases}
   \frac{m_{-} - 1}{\gamma r m_{+} (m_{+} - 1) (m_{+} - m_{-}) A(\underline{L})^{m_{+} - 1}} y^{m_{+}} + \left( \frac{\beta - r + \frac{\theta^{2}}{2}}{\gamma r^{2}} - \frac{\ln (A(\underline{L})) + 1}{\gamma r} \right) y + \frac{y \ln (y)}{\gamma r} \quad \text{if } y \leq A(\underline{L}), \\
   \frac{m_{+} - 1}{\gamma r m_{-} (m_{-} - 1) (m_{+} - m_{-}) A(\underline{L})^{m_{-} - 1}} y^{m_{-}} - \frac{A(\underline{L})}{\gamma \beta} \qquad \qquad \qquad \qquad \qquad \qquad \quad \;\; \text{if } y > A(\underline{L}),
   \end{cases} \\
  \tilde{V}_{\Mer}'(y) & = \tfrac{Y}{r} + \begin{cases}
   \frac{m_{-} - 1}{\gamma r (m_{+} - 1) (m_{+} - m_{-})} \left( \frac{A(\underline{L})}{y} \right)^{1 - m_{+}} + \frac{\beta - r + \frac{\theta^{2}}{2}}{\gamma r^{2}} - \frac{\ln \left( \frac{A(\underline{L})}{y} \right)}{\gamma r} \qquad \quad \text{if } y \leq A(\underline{L}), \\
   \frac{m_{+} - 1}{\gamma r (m_{-} - 1) (m_{+} - m_{-})} \left( \frac{A(\underline{L})}{y} \right)^{1 - m_{-}} \qquad \qquad \qquad \qquad \qquad \qquad \text{if } y > A(\underline{L}),
   \end{cases} \\
  y \tilde{V}_{\Mer}''(y) & = \begin{cases}
   \frac{m_{-} - 1}{\gamma r (m_{+} - m_{-})} \left( \frac{A(\underline{L})}{y} \right)^{1 - m_{+}} + \frac{1}{\gamma r} \qquad \qquad \quad \; \text{if } y \leq A(\underline{L}), \\
   \frac{m_{+} - 1}{\gamma r (m_{+} - m_{-})} \left( \frac{A(\underline{L})}{y} \right)^{1 - m_{-}} \qquad \qquad \qquad \quad \; \text{if } y > A(\underline{L}),
   \end{cases}
   \end{split}
\end{equation}
and $\tilde{V}_{\Mer} \in C^{3}((0, \infty))$ with $\tilde{V}_{\Mer}(y) < \frac{Y}{r} y$, $\tilde{V}_{\Mer}'(y) < \frac{Y}{r}$, $\tilde{V}_{\Mer}''(y) > 0$ for all $y > 0$.
\end{lemma}


\section{Verification and summary} \label{sec:Verification}

\begin{theorem}[Verification of the value function]
The functions $V_{\post}$ in \eqref{post-retirement} and $V_{\pre}$ in \eqref{pre-retirement} are indeed the value functions for post-retirement and pre-retirement satisfying \eqref{Post_value_define} and \eqref{Pre_value_define}, respectively.
\end{theorem}
\begin{proof}
The basic idea follows Lemma 6.3 of Karatzas and Wang \cite{KW}. We first consider the post-retirement case. Let $y = ( - \tilde{V}_{\post}' )^{- 1} (x)$. The inequality in \eqref{Lagrange} become an equality when $\mathbb{E}^{x} \left[ \int_{0}^{\infty} H_{s} c_{s}^{\post} d s \right] = x$. Hence, it suffices to check that $\int_{0}^{\infty} \mathbb{E}^{x} \left[ H_{s}^{2} \left\{ (\pi_{s}^{\post})^{2} + (X_{s}^{\post})^{2} \right\} \right] d s < \infty$ ensuring that $\int_{0}^{t} H_{s} (\sigma \pi_{s}^{\post} - \theta X_{s}^{\post}) d B_{s}$ is a martingale (e.g see \eqref{ito_for_HX_post}). This is shown by 
\begin{align*}
  & \int_{0}^{\infty} \mathbb{E}^{x} \left[ H_{s}^{2}\left\{ (\pi_{s}^{\post})^{2} + (X_{s}^{\post})^{2} \right\} \right] d s = \int_{0}^{\infty} \mathbb{E}^{x} \left[ H_{s}^{2} \left\{ \left( \tfrac{\theta y_{s} \tilde{V}_{\post}''(y_{s})}{\sigma} \right)^{2} +(\tilde{V}_{\post}'(y_{s}))^{2} \right\}\right] d s \\
  & \quad \leq \left( \left( \tfrac{\theta}{\sigma} \right)^{2} \sup_{z > 0} |z^{2} \tilde{V}_{\post}''(z)|^2 + \sup_{z > 0} |z \tilde{V}_{\post}'(z)|^2 \right) \int_{0}^{\infty} \mathbb{E}^{x} \left[ \left( \tfrac{H_{s}}{y_{s}} \right)^{2} \right] d s \\
  & \quad = \left( \left( \tfrac{\theta}{\sigma} \right)^{2} \sup_{z > 0} |z^{2} \tilde{V}_{\post}''(z)|^2 + \sup_{z > 0} |z \tilde{V}_{\post}'(z)|^2 \right) y^{- 2} \int_{0}^{\infty} e^{- 2 \beta s} d s < \infty,
\end{align*}
where \lemref{Post_value_sign} ensures finiteness. Combining this with \lemref{Post_vanish}, we conclude that $V_{\post}$ is indeed the value function for post-retirement. For the pre-retirement case, the equality in \eqref{Budget_constraint_pre} can be shown using a similar argument. Lemma 6.3 of Karatzas and Wang \cite{KW} ensures the existence of a portfolio \( \pi \) satisfying \eqref{admissible_condition} and the equality in \eqref{Budget_constraint_pre}.
\end{proof}
The optimal control problem can be summarized as follows:  
\begin{theorem} \label{Summary}
Let $\tau^{*} = \inf \{ s > 0 : y_{s} \leq \overline{y} \}$ where $\overline{y}$ is defined in \eqref{overline_y}. Then the value function $V$, its dual $\tilde{V}$, the optimal consumption rate $c^{*}$, and the optimal portfolio $\pi^{*}$ are given by
\begin{equation}
    \begin{split}
    V(x) & = \inf_{y > 0} \left(\tilde{V}(y) + y x\right), \\
    \tilde{V}(y) & = \tilde{V}_{\pre}(y) + \mathbb{E} \left[ e^{- \beta \tau^{*}} \tilde{V}_{\post} \left( y e^{\beta \tau^{*}} H_{\tau^{*}} \right) \right], \\
    c^{*}_{t} & = c_{t}^{\pre} \cdot \mathbf{1}_{ \{ t < \tau^{*} \} } + c_{t}^{\post} \cdot \mathbf{1}_{ \{ t \geq \tau^{*} \} }, \\\
    \pi^{*}_{t} & = \pi_{t}^{\pre} \cdot \mathbf{1}_{ \{ t < \tau^{*} \} } + \pi_{t}^{\post} \cdot \mathbf{1}_{ \{ t \geq \tau^{*} \} } ,   
    \end{split}
\end{equation}
and $\tau^{*}$ is the optimal retirement time. Furthermore, the agent never retires if and only if $Y \geq \tfrac{1 - \alpha}{\alpha} \cdot (\overline{L} - \underline{L})$.
\end{theorem}


\section{Properties of the optimal plan} \label{sec:properties}

We now provide a comparison between the benchmark and the optimal plan, analyzing how income $Y$ and the two leisure levels $( \underline{L} ,\,\overline{L})$ affect the optimal portfolio, consumption, and retirement decision.
\begin{proposition}[Effect of Income and Leisure on Optimal Behavior] \label{optimal_properties} 
Let $x$ be fixed. Then {\ \\}
(i) $\overline{y}$ is strictly decreasing in $Y, \underline{L}$, and is strictly increasing in $\overline{L}$. \\
(ii) $c_{*}^{\post}(x)$ and $\pi_{*}^{\post}(x)$ are independent of $(Y, \underline{L}, \overline{L})$. \\
(iii) If $(- \tilde{V}_{\pre}')^{- 1}(x) \geq A(\underline{L})$, then $c_{*}^{\pre}(x) = 0$ and $\pi_{*}^{\pre}(x) = \frac{\theta (1 - m_{-})}{\sigma} \left( x + \frac{Y}{r} \right)$. If $(- \tilde{V}_{\pre}')^{- 1}(x) < A(\underline{L})$, then $c_{*}^{\pre}(x)$ is strictly increasing in $Y$, $\underline{L}$, and is non-increasing in $\overline{L}$. $\pi_{*}^{\pre}(x)$ is strictly decreasing in $\underline{L}$, and is non-decreasing in $\overline{L}$.
\end{proposition}
\begin{proof} {\ \\}
(i) The integrand of \eqref{Equation_overline_y} is strictly increasing in $Y$, $\underline{L}$ and strictly decreasing in $\overline{L}$. Hence, we get the desired results. \medskip

\noindent (ii) Let $I_{(\textnonit)} =\left (- \tilde{V}_{(\textnonit)}'\right)^{- 1}$ where $(\textnonit) \in \{ \Mer, \post, \pre \}$, and define
\begin{align*}
  F(X) := \begin{cases}
   \frac{m_{-} - 1}{\gamma r (m_{+} - 1) (m_{+} - m_{-})} X^{1 - m_{+}} + \frac{\beta - r + \frac{\theta^{2}}{2}}{\gamma r^{2}} - \frac{\ln \left( X \right)}{\gamma r} \quad \text{for } X \geq 1, \\
   \frac{m_{+} - 1}{\gamma r (m_{-} - 1) (m_{+} - m_{-})} X^{1 - m_{-}} \qquad \qquad \qquad \qquad \quad \text{for } X < 1.
   \end{cases}
\end{align*}
From \eqref{Post_value}, we have $x + F \left( \frac{A(\overline{L})}{I_{\post}(x)} \right) = 0$. Since $(m_{+}, m_{-}, \gamma, r)$ are independent of $(Y, \underline{L},\overline{L})$ and 
\begin{align} \label{Derivative_sign_1}
  F'(X) = \begin{cases}
   - \frac{1}{X \gamma r} \left( \frac{m_{-} - 1}{m_{+} - m_{-}} X^{1 - m_{+}} + 1 \right)\quad \text{for } X \geq 1, \\
   - \frac{m_{+} - 1}{\gamma r (m_{+} - m_{-})} X^{- m_{-}} \qquad \qquad \quad \text{for } X < 1,
   \end{cases} < 0,
\end{align}
differentiating $x + F \left( \frac{A(\overline{L})}{I_{\post}(x)} \right) = 0$ with respect to $P \in \{ Y, \underline{L}, \overline{L} \}$ implies $F' \left( \frac{A(\overline{L})}{I_{\post}(x)} \right) \frac{\partial}{\partial P} \left( \frac{A(\overline{L})}{I_{\post}(x)} \right) = 0$. Thus, $\frac{\partial}{\partial P} \left( \frac{A(\overline{L})}{I_{\post}(x)} \right) = 0$. Also since $\frac{\partial}{\partial P} \left( \frac{A(\overline{L})}{I_{\post}(x)} \right) = 0$ and 
\begin{align}
  I_{\post}(x) \tilde{V}_{\post}''(I_{\post}(x)) = \begin{cases}
   \frac{m_{-} - 1}{\gamma r (m_{+} - m_{-})} \left( \tfrac{A(\overline{L})}{I_{\post}(x)} \right)^{1 - m_{+}} + \frac{1}{\gamma r} \qquad \quad \text{if } \tfrac{A(\overline{L})}{I_{\post}(x)} \geq 1, \\
   \frac{m_{+} - 1}{\gamma r (m_{+} - m_{-})} \left( \tfrac{A(\overline{L})}{I_{\post}(x)} \right)^{1 - m_{-}} \qquad \qquad \quad \text{if } \tfrac{A(\overline{L})}{I_{\post}(x)} < 1, 
   \end{cases}
\end{align}
we conclude that  $\frac{\partial \pi_{*}^{\post}}{\partial P} = 0$. \medskip

\noindent (iii) The first statement follows directly from $x + \tilde{V}_{\pre}'(I_{\pre}(x)) = 0$. Now, assume $(- \tilde{V}_{\pre}')^{- 1}(x) < A(\underline{L})$. We aim to show that $c_{*}^{\pre}$ is strictly increasing in $Y$. Differentiating $x + \tilde{V}_{\pre}'(I_{\pre}(x)) = 0$, and using \eqref{Pre_value}, with respect to $Y$ and $\overline{L}$ respectively gives 
\begin{align*}
  \tfrac{\partial I_{\pre}(x)}{\partial Y} = - \tfrac{\frac{1}{r} + \frac{\partial C_{2}}{\partial Y} m_{-} [I_{\pre}(x)]^{m_{-} - 1}}{\tilde{V}_{\pre}''(I_{\pre}(x))} < 0, \quad
  \tfrac{\partial I_{\pre}(x)}{\partial \overline{L}} = - \tfrac{\frac{\partial C_{2}}{\partial \overline{L}} m_{-} [I_{\pre}(x)]^{m_{-} - 1}}{\tilde{V}_{\pre}''(I_{\pre}(x))} \geq 0,
\end{align*}
where $\frac{\partial C_{2}}{\partial Y} = \frac{\partial C_{2}}{\partial \overline{y}} \cdot \frac{\partial \overline{y}}{\partial Y} \leq 0$, $\frac{\partial C_{2}}{\partial \overline{L}} = \frac{\partial C_{2}}{\partial \overline{y}} \cdot \frac{\partial \overline{y}}{\partial \overline{L}} \geq 0$ (due to (i) and $\tilde{C}'(\overline{y}) \geq 0$), and $\tilde{V}_{\pre}'' > 0$ are used in the first and second sign, respectively. This implies $\frac{\partial c_{*}^{\pre}(x)}{\partial Y} > 0$, $\frac{\partial c_{*}^{\pre}(x)}{\partial \overline{L}} \leq 0$.

Combining \eqref{Pre_value} with $x + \tilde{V}_{\pre}'(I_{\pre}(x)) = 0$, we have 
\begin{align} \label{portolio_pre_another}
  I_{\pre}(x) \tilde{V}_{\pre}''(I_{\pre}(x)) = \tfrac{m_{-} - 1}{\gamma r} \left( \tfrac{1}{m_{+} - 1} \left( \tfrac{A(\underline{L})}{I_{\pre}(x)} \right)^{1 - m_{+}} + \ln \left( \tfrac{A(\underline{L})}{I_{\pre}(x)} \right) \right) + (1 - m_{-}) \left( x + \tfrac{Y}{r} + \tfrac{\beta - r + \frac{\theta^{2}}{2}}{\gamma r^{2}} \right) + \tfrac{1}{\gamma r}.
\end{align}
Differentiating \eqref{portolio_pre_another} with respect to $\overline{L}$, we get $sgn \left( \tfrac{\partial \pi_{*}^{\pre}(x)}{\partial \overline{L}} \right) = - sgn \left( \tfrac{\partial c_{*}^{\pre}(x)}{\partial \overline{L}} \right)$ due to $I_{pre}(x) < A(\underline{L})$ and $m_{+} > 1$. Combining this with $\frac{\partial c_{*}^{\pre}(x)}{\partial \overline{L}} \leq 0$, we conclude $\frac{\partial \pi_{*}^{\pre}(x)}{\partial \overline{L}} \geq 0$.

To find sign of $\frac{\partial c_{*}^{\pre}(x)}{\partial \underline{L}}$ and $\frac{\partial \pi_{*}^{\pre}(x)}{\partial \underline{L}}$, we evaluate (by the same way as \eqref{portolio_pre_another}) 
\begin{align}  \label{Pre_c_pi_2}
  sgn \left( \tfrac{\partial \pi_{*}^{\pre}(x)}{\partial \underline{L}} \right) = - sgn \left( \tfrac{\partial c_{*}^{\pre}(x)}{\partial \underline{L}} \right).
\end{align}
From \eqref{Pre_value}, together with $\tilde{V}_{\pre}(\overline{y}) - \frac{Y \overline{y}}{r} = 0 = x + \tilde{V}_{\pre}'(I_{\pre}(x))$, we derive 
\begin{equation} \label{Equation_1}
  \begin{split}
  - C_{2} m_{-} A(\underline{L})^{m_{-} - 1} & = m_{-} \left( \tfrac{A(\underline{L})}{\overline{y}} \right)^{m_{-} - 1} \left( \tfrac{m_{-} - 1}{\gamma r m_{+} (m_{+} - 1) (m_{+} - m_{-})} \left( \tfrac{A(\underline{L})}{\overline{y}} \right)^{1 - m_{+}} + \tfrac{\beta - 2 r + \frac{\theta^{2}}{2}}{\gamma r^{2}} - \tfrac{\ln \left( \frac{A(\underline{L})}{\overline{y}} \right)}{\gamma r} \right) \\
  & = \left( \tfrac{A(\underline{L})}{I_{\pre}(x)} \right)^{m_{-} - 1} \left( x + \tfrac{Y}{r} + \tfrac{m_{-} - 1}{\gamma r (m_{+} - 1) (m_{+} - m_{-})} \left( \tfrac{A(\underline{L})}{I_{\pre}(x)} \right)^{1 - m_{+}} + \tfrac{\beta - r + \frac{\theta^{2}}{2}}{\gamma r^{2}} - \tfrac{\ln \left( \frac{A(\underline{L})}{I_{\pre}(x)} \right)}{\gamma r} \right).
  \end{split}
\end{equation}
For any $X \geq 1$, combining \eqref{Equation_1} with \eqref{Relation_beta_r} gives 
\begin{equation}
  \begin{split}
  & \tfrac{\partial}{\partial X} \left( m_{-} X^{m_{-} - 1} \left( \tfrac{m_{-} - 1}{\gamma r m_{+} (m_{+} - 1) (m_{+} - m_{-})} X^{1 - m_{+}} + \tfrac{\beta - 2 r + \frac{\theta^{2}}{2}}{\gamma r^{2}} - \tfrac{\ln \left( X \right)}{\gamma r} \right) \right) \nonumber \\
  & \quad = \tfrac{m_{-} X^{m_{-} - 2}}{\gamma r} \left( (1 - m_{-}) \left( \tfrac{X^{1 - m_{+}}}{m_{+} (m_{+} - 1)} + \ln (X) + \tfrac{m_{+} m_{-} - 2 m_{+} - 2 m_{-} + 3}{(m_{+} - 1) (m_{-} - 1)} \right) - 1 \right) < 0, \\
  & \tfrac{\partial}{\partial X} \left( X^{m_{-} - 1} \left( x + \tfrac{Y}{r} + \tfrac{m_{-} - 1}{\gamma r (m_{+} - 1) (m_{+} - m_{-})} X^{1 - m_{+}} + \tfrac{\beta - r + \frac{\theta^{2}}{2}}{\gamma r^{2}} - \tfrac{\ln \left( X \right)}{\gamma r} \right) \right) \nonumber \\
  & \quad = (m_{-} - 1) X^{m_{-} - 2} \left( - C_{2} m_{-} A(\underline{L})^{m_{-} - 1} \left( \tfrac{A(\underline{L})}{I_{\pre}(x)} \right)^{1 - m_{-}} \right) + \tfrac{X^{m_{-} - 2}}{\gamma r} \left( \tfrac{1 - m_{-}}{m_{+} - m_{-}} X^{1 - m_{+}} - 1 \right) < 0.
  \end{split}
\end{equation}
Hence, differentiating \eqref{Equation_1} with respect to $\underline{L}$ implies $sgn \left( \frac{\partial}{\partial \underline{L}} \left( \tfrac{A(\underline{L})}{I_{\pre}(x)} \right) \right) = sgn \left( \frac{\partial}{\partial \underline{L}} \left( \tfrac{A(\underline{L})}{\overline{y}} \right) \right)$.

If $\overline{y} \in [A(\overline{L}), A(\underline{L}))$, then differentiating \eqref{oy_small_equation} with respect to $\underline{L}$ implies 
\begin{align*}
  0 & = \tfrac{\left( \frac{A(\underline{L})}{\overline{y}} \right)^{- 1}}{\gamma m_{+} (m_{+} - 1)} \left( \tfrac{A(\overline{L})}{A(\underline{L})} \left( \tfrac{A(\underline{L})}{\overline{y}} \right)^{m_{+}} (m_{+} - 1) + 1 - m_{+} \left( \tfrac{A(\underline{L})}{\overline{y}} \right)^{m_{+} - 1} \right) \tfrac{\partial}{\partial \underline{L}} \left( \tfrac{A(\underline{L})}{\overline{y}} \right) + \tfrac{\left( \frac{A(\underline{L})}{\overline{y}} \right)^{m_{+}} \frac{A(\overline{L})}{A(\underline{L})} (1 - \alpha)}{m_{+} \alpha},
\end{align*}
where the equality is derived by eliminating $Y$. Since 
$$Z \mapsto \tfrac{A(\overline{L})}{A(\underline{L})} \left( \tfrac{A(\underline{L})}{Z} \right)^{m_{+}} (m_{+} - 1) + 1 - m_{+} \left( \tfrac{A(\underline{L})}{Z} \right)^{m_{+} - 1} =: F(Z)$$ is strictly decreasing in $Z \in [A(\overline{L}), A(\underline{L}))$ with $\lim_{Z \uparrow A(\underline{L})} F(Z) < 0$, we get $\frac{\partial c_{*}^{\pre}(x)}{\partial \underline{L}} > 0$. Together with \eqref{Pre_c_pi_2}, we conclude $\frac{\partial \pi_{*}^{\pre}(x)}{\partial \underline{L}} < 0$.

If $0 < \overline{y} < A(\overline{L})$, on the other hand, then differentiating \eqref{oy_small_equation} with respect to $\underline{L}$ implies 
\begin{align*}
  0 = \left( \tfrac{A(\underline{L})}{\overline{y}} \right)^{m_{+} - 2} \left( Y - \tfrac{1 - \alpha}{\alpha} (\overline{L} - \underline{L}) \right) \tfrac{\partial}{\partial \underline{L}} \left( \tfrac{A(\underline{L})}{\overline{y}} \right) + \tfrac{1 - \alpha}{\alpha (m_{+} - 1)} \left( \left( \tfrac{A(\underline{L})}{\overline{y}} \right)^{m_{+} - 1} - \left( \tfrac{A(\underline{L})}{A(\overline{L})} \right)^{m_{+} - 1} \right).
\end{align*}
Due to $m_{+} > 1$, $A(\overline{L}) < A(\underline{L})$, and $Y < \tfrac{1 - \alpha}{\alpha} (\overline{L} - \underline{L})$, the above equality implies $\frac{\partial c_{*}^{\pre}(x)}{\partial \underline{L}} > 0$. Together with \eqref{Pre_c_pi_2}, we conclude $\frac{\partial \pi_{*}^{\pre}(x)}{\partial \underline{L}} < 0$.
\end{proof}

\begin{remark}[Economic interpretations]
We provide economic interpretations of \proref{optimal_properties} as follows: \\
The timing of retirement is influenced by the interplay between leisure and income. When income or leisure before retirement increases, the agent finds working more attractive and tends to delay retirement, opting to enjoy the benefits of continued income and utility from leisure while still employed. Conversely, if leisure after retirement becomes more desirable, the agent has a stronger incentive to retire earlier to maximize time spent in this enhanced state of leisure. This balance between pre- and post-retirement utilities drives the optimal retirement decision.

Before retirement, the agent's optimal consumption and portfolio choices are independent of leisure and income levels. This independence indicates that these financial decisions are shaped more by preferences for risk and consumption timing than by external factors like the agent's income or leisure availability. The agent's focus is on achieving an optimal financial trajectory, which remains unaffected by variations in income or leisure.

Income and leisure exert nuanced effects on pre-retirement behavior. An increase in income leads to higher consumption, as the agent can afford to allocate more resources to present utility. Greater leisure before retirement motivates the agent to reduce stock investments, prioritizing time spent on leisure activities, which indirectly increases consumption. However, when leisure after retirement becomes more significant, the agent adjusts by investing more in stocks to secure higher returns for the future, which leads to a reduction in current consumption. These adjustments reflect the agents strategy to balance present and future utility optimally.
\end{remark}

\begin{figure}
		\begin{center}
			\includegraphics[width=0.45\textwidth]{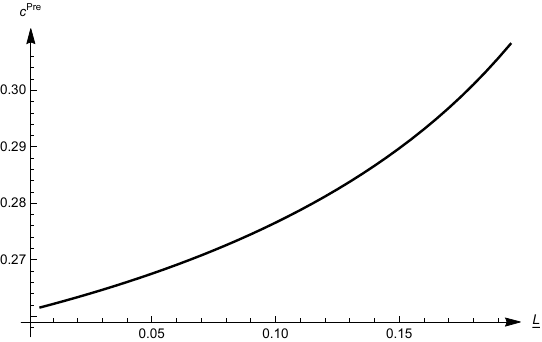} \quad
			\includegraphics[width=0.45\textwidth]{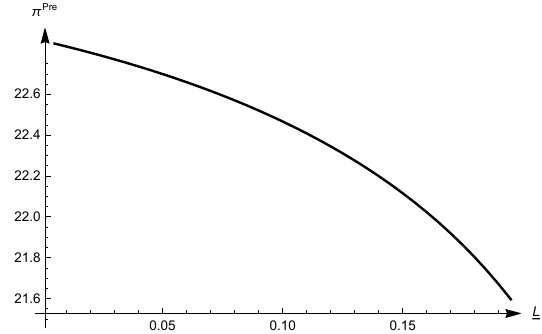}
	\end{center}
	 \caption{The two graphs describe optimal consumption and portfolio as function of $\underline{L}$. The parameters are $\beta = 0.03$, $r = 0.01$, $\mu = 0.07$, $\sigma = 0.2$, $\gamma = 3$, $\overline{L} = 0.5$, $Y = 0.1$, $\alpha = 0.4$, $x = 1$.
		}
		\label{various_uL}
\end{figure}

\begin{figure}
		\begin{center}
			\includegraphics[width=0.45\textwidth]{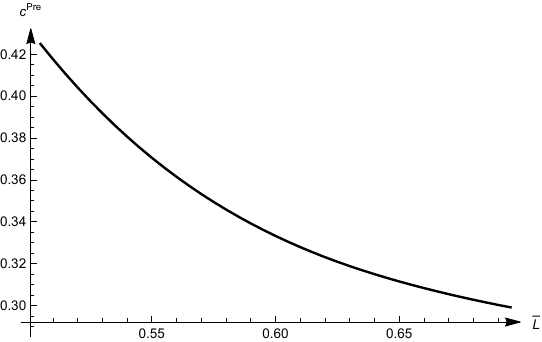} \quad
			\includegraphics[width=0.45\textwidth]{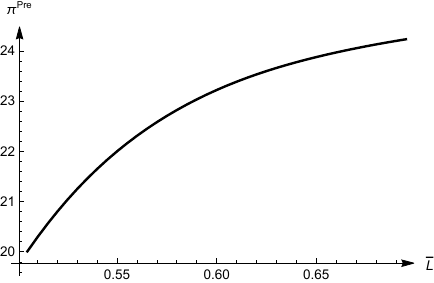}
	\end{center}
	 \caption{The two graphs describe optimal consumption and portfolio of pre-retirement as function of $\overline{L}$. The parameters are $\beta = 0.03$, $r = 0.01$, $\mu = 0.07$, $\sigma = 0.2$, $\gamma = 3$, $\underline{L} = 0.3$, $Y = 0.1$, $\alpha = 0.4$, $x = 1$.
		}
		\label{various_oL}
\end{figure}

\begin{remark}
Equation \eqref{portolio_pre_another} shows that
\begin{align*}
  \tfrac{\partial}{\partial Y} \left( I_{\pre}(x) \tilde{V}_{\pre}''(I_{\pre}(x)) \right) = \tfrac{1 - m_{-}}{\gamma r} \left( \tfrac{A(\underline{L})}{I_{\pre}(x)} \right)^{- 1} \left( \left( \tfrac{A(\underline{L})}{I_{\pre}(x)} \right)^{1 - m_{+}} - 1 \right) \tfrac{\partial}{\partial Y} \left( \tfrac{A(\underline{L})}{I_{\pre}(x)} \right) + \tfrac{1 - m_{-}}{r},
\end{align*}
while $\frac{\partial c_{*}^{\pre}(x)}{\partial Y} > 0$ holds for $I_{\pre}(x) < A(\underline{L})$, the term $\frac{1 - m_{-}}{r}$ can potentially disrupt monotonicity. This nuanced behavior is illustrated in Figure \ref{various_Y}.
\end{remark}

\begin{figure}
		\begin{center}
			\includegraphics[width=0.45\textwidth]{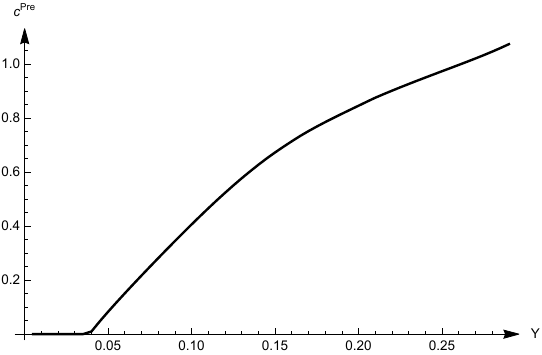} \quad 
			\includegraphics[width=0.45\textwidth]{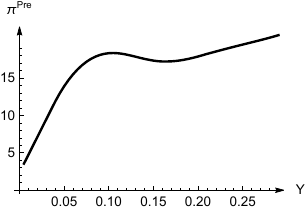} \\
			\includegraphics[width=0.45\textwidth]{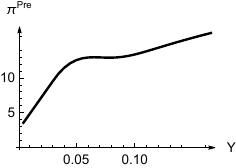} \quad 
			\includegraphics[width=0.45\textwidth]{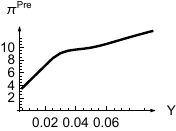}
	\end{center}
	 \caption{The first graph shows $c_{*}^{\pre}$ and the remaining graphs show $\pi_{*}^{\pre}$ as a function of $Y$. The parameters are $\beta = 0.03$, $r = 0.01$, $\mu = 0.07$, $\sigma = 0.2$, $\gamma = 3$, $\underline{L} = 0.4$, $\overline{L} = 0.5$, $x = 1$ with $\alpha = 0.25$ (in $c_{*}^{\pre}$), $\alpha = 0.25$, $\alpha = 0.37$, $\alpha = 0.5$ (in $\pi_{*}^{\pre}$) in order.
		}
		\label{various_Y}
\end{figure}

\begin{proposition}[Comparison at same wealth level] \label{Compare} Let us define 
\begin{align*}
  \begin{cases}
  c_{*}^{\Mer}(x) := \tfrac{1}{\gamma} \left( \ln \left( \tfrac{A(\underline{L})}{I_{\Mer}(x)} \right) \right)^{+}, \\\
  c_{*}^{\pre}(x) := \tfrac{1}{\gamma} \left( \ln \left( \tfrac{A(\underline{L})}{I_{\pre}(x)} \right) \right)^{+}, \\
  c_{*}^{\post}(x) := \tfrac{1}{\gamma} \left( \ln \left( \tfrac{A(\overline{L})}{I_{\post}(x)} \right) \right)^{+},
  \end{cases}
  \quad 
  \begin{cases}  
  \pi_{*}^{\Mer}(x) := \tfrac{\theta I_{\Mer}(x) \tilde{V}_{\Mer}''( I_{\Mer}(x) )}{\sigma}, \\
  \pi_{*}^{\pre}(x) := \tfrac{\theta I_{\pre}(x) \tilde{V}_{\pre}''( I_{\pre}(x) )}{\sigma}, \\
  \pi_{*}^{\post}(x) := \tfrac{\theta I_{\post}(x) \tilde{V}_{\post}''( I_{\post}(x) )}{\sigma},
  \end{cases}
\end{align*}
where the optimal consumption and portfolio are observed at each wealth level $x$. Specifically, $c_{*}^{\Mer}(x)$ and $\pi_{*}^{\Mer}(x)$ are defined in $( - \frac{Y}{r}, \infty)$, $c_{*}^{\post}(x)$ and $\pi_{*}^{\post}(x)$ are defined in $( 0, \infty )$, and $c_{*}^{\pre}(x)$ are $\pi_{*}^{\pre}(x)$ are defined in $( - \frac{Y}{r}, - \tilde{V}_{\pre}'(\overline{y})]$. Then for each $x$, \\
\noindent (i) 
\begin{align} \label{c_compare}
  c_{*}^{\post}(x) \leq c_{*}^{\Mer}(x) \quad \text{and} \quad c_{*}^{\pre}(x) \leq c_{*}^{\Mer}(x).
\end{align}
\noindent (ii) 
\begin{align}
  \begin{cases}
  x \leq - \tilde{V}_{\Mer}'(A(\underline{L})) \quad & \Rightarrow \quad \pi_{*}^{\post}(x) < \pi_{*}^{\Mer}(x) = \pi_{*}^{\pre}(x) = (1 - m_{-}) (x + \frac{Y}{r}) \\
  x > - \tilde{V}_{\Mer}'(A(\underline{L})) \quad & \Rightarrow \quad \pi_{*}^{\post}(x) < \pi_{*}^{\Mer}(x) < \pi_{*}^{\pre}(x).
  \end{cases}\label{Pi_compare}
\end{align}
\end{proposition}
\begin{proof} 
\noindent (i) It suffices to verify the following two inequalities:
\begin{align}  \label{Inequality_1}
  I_{\Mer}(x) < I_{\pre}(x),
\end{align}
and 
\begin{align}  \label{Inequality_2}
  \tfrac{A(\overline{L})}{I_{\post}(x)} < \tfrac{A(\overline{L})}{I_{\Mer}(x)}.
\end{align}
Since $\tilde{V}_{\Mer}''$, $\tilde{V}_{\pre}''$, $\tilde{V}_{\post}''$ are positive, \eqref{Inequality_1} is equivalent to 
\begin{align}  \label{Inequality_1_equiv}
  \tilde{V}_{\pre}' < \tilde{V}_{\Mer}'.
\end{align}
This equivalence follows from \eqref{Pre_value} and \eqref{Mer_value}, where $C_2 > 0$, thus proving \eqref{Inequality_1}.

To verify \eqref{Inequality_2}, let $F(X) := \frac{m_{-} - 1}{\gamma r (m_{+} - 1) (m_{+} - m_{-})} X^{1 - m_{+}} + \frac{\beta - r + \frac{\theta^{2}}{2}}{\gamma r^{2}} - \frac{\ln \left( X \right)}{\gamma r}$. From the condition $0 = x + \tilde{V}_{\post}'(I_{\post}(x)) = x + \tilde{V}_{\Mer}'(I_{\Mer}(x))$, we obtain $0 = x + F \left( \tfrac{A(\overline{L})}{I_{\post}(x)} \right) = x + \frac{Y}{r} + F \left( \tfrac{A(\underline{L})}{I_{\Mer}(x)} \right)$. Using \eqref{Derivative_sign_1} and $Y > 0$, inequality \eqref{Inequality_2} follows. Thus, we have proven \eqref{c_compare}. \medskip

\noindent (ii) It is straightforward to verify the following inequality:
\begin{align}  \label{Inequality_3}
  \tilde{V}_{\post}'' < \tilde{V}_{\Mer}'' < \tilde{V}_{\pre}''.
\end{align}
This follows from \eqref{Post_value}, \eqref{Pre_value}, and \eqref{Mer_value}, combined with $A(\overline{L}) < A(\underline{L})$ and $C_{2} > 0$. Additionally, since $\tilde{V}_{\Mer} \in C^{3}((0, \infty))$, we can compute
\begin{equation}  \label{Inequality_4}
  \begin{split}
  \tfrac{d}{d y} \left( y \tilde{V}_{\Mer}''(y) \right) = \begin{cases}
  \tfrac{(m_{-} - 1) (m_{+} - 1)}{\gamma r (m_{+} - m_{-}) A(\underline{L})^{m_{+} - 1}} y^{m_{+} - 2} \quad \text{ if } y \leq A(\underline{L}) \\
    \tfrac{(m_{-} - 1) (m_{+} - 1)}{\gamma r (m_{+} - m_{-}) A(\underline{L})^{m_{-} - 1}} y^{m_{-} - 2} \quad \text{ if } y > A(\underline{L})
  \end{cases} < 0.
  \end{split}
\end{equation}
This yields the conclusion
\begin{align}  \label{Pi_inequality_1}
  & I_{\post}(x) \tilde{V}_{\post}'' \left( I_{\post}(x) \right) < I_{\post}(x) \tilde{V}_{\Mer}'' \left( I_{\post}(x) \right) < I_{\Mer}(x) \tilde{V}_{\Mer}'' \left( I_{\Mer}(x) \right),
\end{align}
where the first inequality in \eqref{Pi_inequality_1} uses \eqref{Inequality_3}, and the second inequality uses \eqref{Inequality_1} and \eqref{Inequality_4}. Therefore, we conclude that $\pi_{*}^{\post}(x) < \pi_{*}^{\Mer}(x)$.

Next, let us consider \eqref{Post_value} and \eqref{Mer_value}. For $x \leq - \tilde{V}_{\Mer}'(A(\underline{L}))$, we get $\pi_{*}^{\Mer}(x) = \pi_{*}^{\pre}(x) = (1 - m_{-}) (x + \frac{Y}{r})$, which follows from $0 = x + \tilde{V}_{\pre}'(I_{\pre}(x)) = x + \tilde{V}_{\Mer}'(I_{\Mer}(x))$. Thus, we restrict our attention to $x > - \tilde{V}_{\Mer}'(A(\underline{L}))$ and define the following function:
\begin{align*}
  \tilde{W}(I, p) := \tfrac{Y I}{r} + \begin{cases}
   \tfrac{m_{-} - 1}{\gamma r m_{+} (m_{+} - 1) (m_{+} - m_{-}) A(\underline{L})^{m_{+} - 1}} I^{m_{+}} + \tilde{C}(p) I^{m_{-}} + \left( \frac{\beta - r + \frac{\theta^{2}}{2}}{\gamma r^{2}} - \frac{\ln (A(\underline{L})) + 1}{\gamma r} \right) I + \frac{I \ln (I)}{\gamma r} \\
   \qquad \qquad \qquad \qquad \qquad \qquad \qquad \qquad \qquad \qquad \qquad \quad \;\; \text{if } p < I \leq A(\underline{L}), \\
   C_{3} I^{m_{-}} - \frac{A(\underline{L})}{\gamma \beta} \qquad \qquad \qquad \qquad \qquad \qquad \qquad \qquad \quad \; \text{if } I > A(\underline{L}),
   \end{cases}
\end{align*}
where 
\begin{align*}
  \tilde{C}(p) := \tfrac{r [ \ln (A(\underline{L})) - \ln (p) ] - \left( \beta - 2 r + \frac{\theta^{2}}{2} \right)}{\gamma r^{2} p^{m_{-} - 1}} - C_{1} p^{m_{+} - m_{-}}.
\end{align*}
The above form is considered as the ODE
\begin{align*}
    \begin{cases}
  \tfrac{\theta^{2} I^{2}}{2} \cdot \tilde{W}_{I I}(I, p) + (\beta - r) I \tilde{W}_{I}(I, p) - \beta \tilde{W}(I, p) + \underline{U}(I) = 0 \quad \forall I > p, \\
  \tilde{W}(I, p) = 0 \quad \forall \, 0 < I \leq p,
    \end{cases}
\end{align*}
and the stochastic representation is 
\begin{align} \label{Tilde_W}
  \tilde{W}(I, p) = \tfrac{Y I}{r} + \mathbb{E}^{x} \left[ \int_{0}^{\inf \{ u > 0 : y_{u} \leq p \}} e^{- \beta s} \underline{U}(y_{s}) d s \right].
\end{align}
From this, we can find that $\tilde{W}(I, p) \in C^{2}((0, \infty) \setminus \{ p \}) \cap C((0, \infty))$ for fixed $p$. Thus, the following holds:
\begin{align*}
  \tilde{W}_{I}(I, p) & = \tfrac{Y}{r} + \begin{cases}
   \tfrac{m_{-} - 1}{\gamma r (m_{+} - 1) (m_{+} - m_{-}) A(\underline{L})^{m_{+} - 1}} I^{m_{+} - 1} + \tilde{C}(p) m_{-} I^{m_{-} - 1} + \frac{\beta - r + \frac{\theta^{2}}{2}}{\gamma r^{2}} + \frac{\ln (I) - \ln (A(\underline{L}))}{\gamma r} \quad \text{if } p < I \leq A(\underline{L}), \\
   \left( m_{-} \tilde{C}(p) + \tfrac{m_{+} - 1}{\gamma r (m_{-} - 1) (m_{+} - m_{-}) A(\underline{L})^{m_{-} - 1}} \right) I^{m_{-} - 1} \qquad \qquad \qquad \qquad \qquad \qquad \qquad \text{if } I > A(\underline{L})
   \end{cases}, \\
   \tilde{W}_{I I}(I, p) & = \begin{cases}
   \tfrac{1}{I} \left( \tfrac{m_{-} - 1}{\gamma r (m_{+} - m_{-}) A(\underline{L})^{m_{+} - 1}} I^{m_{+} - 1} + \tilde{C}(p) m_{-} (m_{-} - 1) I^{m_{-} - 1} + \frac{1}{\gamma r} \right) \quad \text{if } p < I \leq A(\underline{L}), \\
   \left( m_{-} (m_{-} - 1) \tilde{C}(p) + \tfrac{m_{+} - 1}{\gamma r (m_{+} - m_{-}) A(\underline{L})^{m_{-} - 1}} \right) I^{m_{-} - 2} \qquad \qquad \qquad \quad \text{if } I > A(\underline{L}).
   \end{cases}
\end{align*}
Let a function $(x, p) \mapsto I(x, p)$ be the inverse function of $- \tilde{W}_{I}$, which is well-defined because $\tilde{W}_{I I} > 0$ (as in the pre-retirement case). One can then check that $\tilde{V}_{\Mer}'(I(x, 0)) = \tilde{W}_{I}(I(x, 0), 0)$, $\tilde{V}_{\pre}'(I(x, \overline{y})) = \tilde{W}_{I}(I(x, \overline{y}), \overline{y})$, and $\tilde{C}'(p) \geq 0 = \tilde{C}(0)$ with positive measure of the set $\{ q : \tilde{C}'(q) > 0 \}$. This follows from \eqref{Tilde_W}, the strictly decreasing function  $p \mapsto \inf \{ u > 0 : y_{u} \leq p \}$, and $\tilde{C} \in C^{2}((0, \infty))$.

To show $\pi_{*}^{\Mer}(x) < \pi_{*}^{\pre}(x)$, it is enough to show that $I(x, \overline{y}) \tilde{W}_{I I}(I(x, \overline{y}), \overline{y}) > I(x, 0) \tilde{W}_{I I}(I(x, 0), 0)$. Differentiate the equation $x + \tilde{W}_{I}(I(x, p), p) = 0$ with respect to $p$ to get 
\begin{align}  \label{DI_equation}
  0 = \tilde{W}_{I I}(I(x, p), p) \tfrac{\partial I(x, p)}{\partial p} + \tilde{C}'(p) m_{-} I(x, p)^{m_{-} - 1}.
\end{align}

On the other hand, we have
\begin{align*}
  & \tfrac{\partial}{\partial p} \left( I(x, p) \tilde{W}_{I I}(I(x, p), p) \right) = \tfrac{\partial}{\partial I} \left( I \tilde{W}_{I I}(I, p) \right) \Big|_{I = I(x, p)} \cdot \tfrac{\partial I(x, p)}{\partial p} + \tilde{C}'(p) m_{-} (m_{-} - 1) I(x, p)^{m_{-} - 1} \\
  & \quad = m_{-} I(x, p)^{m_{-} - 1} \tilde{C}'(p) \cdot \tfrac{(m_{-} - 1) \tilde{W}_{I I}(I, p) - \frac{\partial}{\partial I} \left( I \tilde{W}_{I I}(I, p) \right)}{\tilde{W}_{I I}(I, p)} \Big|_{I = I(x, p)},
\end{align*}
where \eqref{DI_equation} is used in the second equality. Given that $I(x, p) < A(\underline{L})$ (since we restricted $x > - \tilde{V}_{\Mer}'(A(\underline{L}))$), the positive measure of the set $\{ q : m_{-} \tilde{C}'(q) < 0 \}$, $\tilde{W}_{I I}(I, p) > 0$, and 
\begin{align*}
  (m_{-} - 1) \tilde{W}_{I I}(I, p) \Big|_{I = I(x, p)} - \tfrac{\partial}{\partial I} \left( I \tilde{W}_{I I}(I, p) \right) \Big|_{I = I(x, p)} = \tfrac{m_{-} - 1}{\gamma r I(x, p)} \left( 1 - \left( \tfrac{I(x, p)}{A(\underline{L})} \right)^{m_{+} - 1} \right) < 0,
\end{align*}
we conclude
\begin{align*}
  I(x, \overline{y}) \tilde{W}_{I I}(I(x, \overline{y}), \overline{y}) - I(x, 0) \tilde{W}_{I I}(I(x, 0), 0) = \int_{0}^{\overline{y}} \tfrac{\partial}{\partial p} \left( I(x, p) \tilde{W}_{I I}(I(x, p), p) \right) d p > 0,
\end{align*}
which gives the desired result $ \pi_{*}^{\Mer}(x) < \pi_{*}^{\pre}(x) $, or equivalently, \eqref{Pi_compare}.
\end{proof}

\begin{figure}
		\begin{center}
			\includegraphics[width=0.45\textwidth]{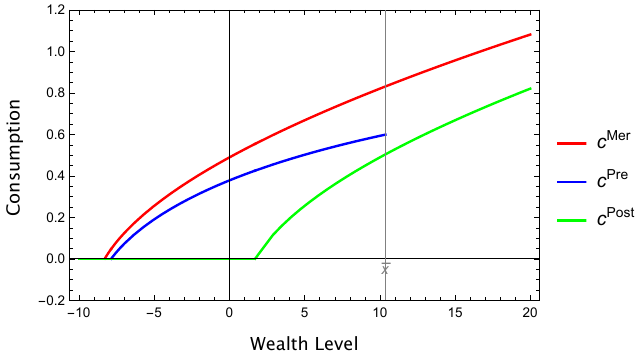} \quad
			\includegraphics[width=0.45\textwidth]{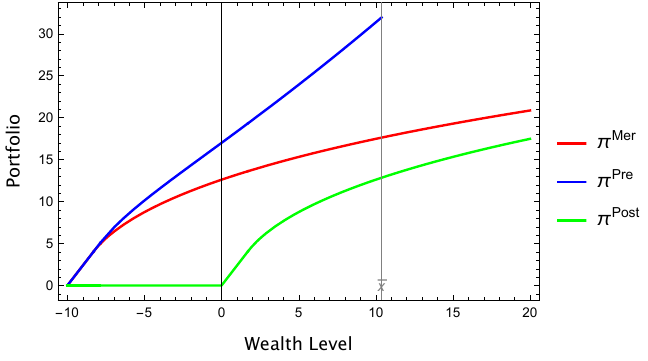}
	\end{center}
	 \caption{The two graphs describe optimal consumption and portfolio as a function of total wealth level. The parameters are $\beta = 0.03$, $r = 0.01$, $\mu = 0.07$, $\sigma = 0.2$, $\gamma = 3$, $\underline{L} = 0.3$, $\overline{L} = 0.5$, $Y = 0.1$, $\alpha = 0.4$.
		}
		\label{optimal_consumption_and_portfolio}
\end{figure}

\begin{remark}
When $Y$ and $\overline{L} - \underline{L}$ are small, $c_{*}^{\pre}(x) < c_{*}^{\post}(x)$, contrary to expectations. This arises due to reduced resources and uncertainty in the pre-retirement phase, as shown in Figure \ref{optimal_consumption_and_portfolio_1}.
\end{remark}

\begin{figure}
		\begin{center}
			\includegraphics[width=0.45\textwidth]{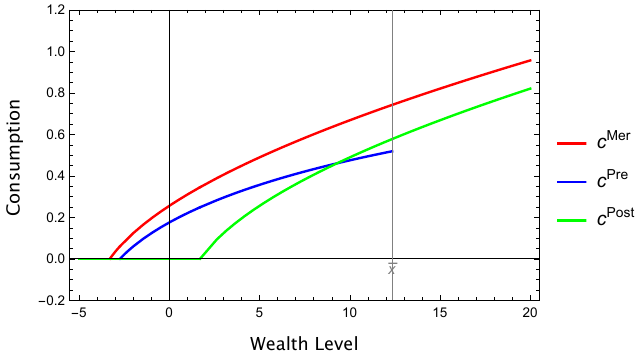}
	\end{center}
	 \caption{The graph describes optimal consumption as a function of total wealth level. The parameters are $\beta = 0.03$, $r = 0.01$, $\mu = 0.07$, $\sigma = 0.2$, $\gamma = 3$, $\underline{L} = 0.4$, $\overline{L} = 0.5$, $Y = 0.05$, $\alpha = 0.4$.
		}
		\label{optimal_consumption_and_portfolio_1}
\end{figure}


\section{Conclusions} \label{sec:conclusions}

This paper investigates an optimal control problem involving investment, consumption, and retirement decisions under exponential (CARA-type) utility. Our main contribution lies in analyzing the equivalent condition for no retirement and examining how income $ Y $ and the two leisure levels, $\underline{L}$ and $\overline{L}$, influence optimal portfolio, consumption, and retirement decisions. Specifically, the condition for no retirement is given by $Y \geq \frac{1 - \alpha}{\alpha} (\overline{L} - \underline{L})$, which defines the threshold income level needed to avoid retirement. The study reveals that the optimal portfolio and consumption decisions before retirement are independent of both income and the two leisure levels. In contrast, post-retirement decisions for the optimal portfolio and consumption demonstrate opposing monotonic relationships with leisure. Specifically, for a given $x$, $c_{*}^{\pre}(x)$ is non-decreasing in $\underline{L}$ and non-increasing in $\overline{L}$, while $\pi_{*}^{\pre}(x)$ shows the opposite trend, being non-increasing in $\underline{L}$ and non-decreasing in $\overline{L}$. These results underscore the distinct effects that leisure before and after retirement have on the agent's financial strategies. While $c_{*}^{\pre}(x)$ is shown to increase with income $Y$, the monotonic behavior of $\pi_{*}^{\pre}(x)$ with respect to $Y$ remains indeterminate. This ambiguity highlights the nuanced interplay between income and investment decisions, reflecting the complex nature of optimal control in this setting. These findings provide valuable insights into how individuals can balance their financial and leisure preferences across different life stages.

\bibliographystyle{siam}
\bibliography{draft}

\end{document}